\documentclass[a4paper,12pt]{article}

\usepackage[pdftex]{graphicx}
\usepackage[section] {placeins}
\usepackage{enumerate}
\usepackage[normalem]{ulem}
\usepackage{amsmath}
\usepackage{amsthm}
\usepackage{latexsym}
\usepackage{amsfonts}
\usepackage{amssymb}

\usepackage{tikz-cd}
\usetikzlibrary{graphs}
\usetikzlibrary{arrows.meta}
\usetikzlibrary{shapes.geometric}

\usepackage{authblk}

\usepackage{latexsym,amssymb,amsfonts, amsthm, amsmath}
\usepackage[english]{babel}
\usepackage{bm}
\usepackage{mathrsfs}

\newtheorem{thm}{Theorem}[section]
\newtheorem{lem}[thm]{Lemma}

\newtheorem{conj}[thm]{Conjecture}

\theoremstyle{definition} 
\newtheorem{defn}[thm]{Definition} 
\newtheorem{exa}[thm]{Example}

\newcommand{\ub}{\underbrace}

\def\BHR{\mathop{\rm BHR}}

\usepackage{fullpage}

\begin{document}

\title{Growable Realizations: a Powerful Approach to the Buratti-Horak-Rosa Conjecture}

\author[1]{M.~A.~Ollis}
\author[2]{Anita Pasotti} 
\author[3]{Marco A.~Pellegrini}
\author[4]{John R.~Schmitt}

\affil[1]{Marlboro Institute for Liberal Arts and Interdisciplinary Studies, Emerson College, Boston, MA 02116, USA} 
\affil[2]{DICATAM, Sez.~Matematica, Universit\`a degli Studi di Brescia, Via Branze~43, I~25123 Brescia, Italy}
\affil[3]{Dipartimento di Matematica e Fisica, Universit\`a Cattolica del Sacro Cuore, Via Musei~41, I~25121 Brescia, Italy}
\affil[4]{Mathematics Department, Middlebury College, Middlebury, VT 05753, USA}

\maketitle

\begin{abstract}
Label the vertices of the complete graph $K_v$ with the integers $\{ 0, 1, \ldots, v-1 \}$ and define the {\em length} of the edge between $x$ and $y$ to be $\min( |x-y| , v - |x-y|  )$.   Let $L$ be a multiset of size $v-1$ with underlying set contained in $\{ 1, \ldots, \lfloor v/2 \rfloor \}$.  The Buratti-Horak-Rosa Conjecture is that there is a Hamiltonian path in $K_v$ whose edge lengths are exactly $L$ if and only if for any divisor $d$ of $v$ the number of multiples of $d$ appearing in $L$ is at most $v-d$.

We introduce ``growable realizations," which enable us to prove many new instances of the conjecture and to reprove known results in a simpler way.  
As examples of the new method, we give a complete solution when the underlying set is contained in $\{ 1,4,5 \}$ or in $\{ 1,2,3,4 \}$ and a partial result when the underlying set has the form $\{ 1, x, 2x \}$. We believe that 
for any set~$U$ of positive integers there is a finite set of growable realizations that implies the truth of the Buratti-Horak-Rosa Conjecture for all but finitely many multisets with underlying set~$U$.

\end{abstract}

\section{Introduction}\label{sec:intro}

Let $K_v$ be the complete graph on $v$ vertices, labeled with the integers $\{ 0,1, \ldots, v-1 \}$.  For two vertices $x$ and $y$, define the {\em length} of the edge between them to be 
$$\ell(x,y) = \min( |x-y| , \ v - |x-y|  ),$$ 
which is an integer in the range $1 \leq \ell(x,y) \leq \lfloor v/2 \rfloor$.

A Hamiltonian path~${\bm h} = [h_1, h_2, \ldots, h_{v}]$ in $K_v$ uses $v-1$ edges and gives a multiset~$L = \{ \ell(h_i, h_{i+1}) : 1 \leq i \leq v-1 \}$ of edge-lengths.  Call ${\bm h}$ a {\em realization} of $L$ or say that ${\bm h}$ {\em realizes} $L$.  For example, with $v=7$ the Hamiltonian path $[0,5,1,2,6,3,4]$ has edge-length sequence $[2,3,1,3,3,1]$ and hence realizes the multiset $\{ 1^2, 2, 3^3 \}$ (where exponents indicate multiplicity).

Given a multiset~$L$, its underlying set is given by~$U =  \{ x : x \in L \}$.

The focus of our inquiry is the {\em Buratti-Horak-Rosa Conjecture}, or {\em BHR Conjecture}:

\begin{conj}\label{conj:bhr}
Let $L$ be a multiset of size $v-1$ with underlying set~$U$ contained in $\{ 1, \ldots, \lfloor v/2 \rfloor \}$.  Then there is a realization of~$L$ in $K_v$ if and only if for any divisor~$d$ of~$v$ the number of multiples of~$d$ in~$L$ is at most~$v-d$.
\end{conj}

When~$v$ is prime, in which case the condition on divisors is always satisfied, we have the original {\em Buratti Conjecture}, see~\cite{BM13,West}.  Horak and Rosa~\cite{HR09} generalize this to composite~$v$ and show that the condition on divisors is necessary; Pasotti and Pellegrini~\cite{PP14b} reformulate Horak and Rosa's statement into the one in Conjecture~\ref{conj:bhr}.  

Call a multiset~$L$ of size $v-1$ {\em admissible} if it has underlying set $U \subseteq \{ 1, \ldots, \lfloor v/2 \rfloor \}$ and it satisfies the divisor condition of the BHR conjecture.  Denote the BHR Conjecture for~$L$ by $\BHR(L)$.

Much work has been done on the BHR Conjecture.  Theorem~\ref{th:known} captures the main progress that has been made to date.

\begin{thm}\label{th:known}
Let $L$ be a multiset of size~$v-1$ with underlying set~$U$.  In each of the following cases, if $L$ is admissible, then it is realizable.
\begin{enumerate}
\item $|U| \leq 2$ \cite{DJ09,HR09},
\item $U = \{1,2,4\}, \{1,2,6\}, \{1,2,8 \}$ \cite{PP14}, 
\item $U \subseteq \{1,2,3,5\}$ \cite{CD10,PP14b}, 
\item $L = \{1^a, 2^b, 3^c, 4^d \}$ with either $a \geq 3$ and $c,d \geq 1$ or $a=2$ and $b,c,d \geq 1$ \cite{OPPS},
\item $L = \{1^a,2^b, x^c  \}$ when $x$ is even and $a+b \geq x-1$ \cite{PP14},
\item $L = \{1^a, x^b, (x+1)^c \}$ when $x$ is odd and either $a \geq \min( 3x-3, b+2x-3)$ or $a \geq 2x-2$ and  $c \geq 4b/3$ \cite{OPPS},
\item $L = \{1^a, x^b, (x+1)^c \}$ when $x$ is even and either $a \geq \min( 3x-1, c+2x-1)$ or $a \geq 2x-1$ and  $b \geq c$ \cite{OPPS},
\item $U \subseteq \{1, 2, 4, \ldots, 2x \}$  and $\{1^{2x-1}, 2x \} \subseteq L$ \cite{OPPS}, 
\item $U \subseteq \{1, 2, 4, \ldots, 2x, 2x+1 \}$  and $\{1^{6x-1}, 2x+1 \} \subseteq L$ \cite{OPPS},
\item $L = \{ 1^{a_1}, 2^{a_2}, \ldots, x^{a_x} \}$ with $a_1 \geq a_2 \geq \cdots \geq a_x$ \cite{Monopoli15,OPPS},
\item $L = M \cup \{1^a\}$ for any multiset~$M$ and~$a > a_M$, where $a_M$ is a constant that depends on~$M$ \cite{HR09},
\item $v \leq 19$ or $v=23$ \cite{Meszka}.
\end{enumerate}
\end{thm}

After proving Theorem~\ref{th:known}.11, Horak and Rosa observe that ``to get an explicit bound... one only needs refer to lemmas used in the proof"~\cite{HR09}.  It turns out that their methodology can be used to give a bound that is linear in the elements of the underlying set and independent of their multiplicities, neither of which is clear from the statement of the result.  We believe that this is of interest and so give an explicit bound with these properties in Theorem~\ref{th:hr}.   

\begin{thm}\label{th:hr}
Let $M$ be a multiset with underlying set $U = \{ x_1, \ldots, x_k \}$, where $1 < x_1 < \cdots < x_k$.  Then $L = M \cup \{1^s \}$ is realizable for any $s \geq 3x_k - 5 + \sum_{i=1}^{k} x_i$.
\end{thm}

\begin{proof}[Proof Outline]
We give the steps required to establish the bound, referring to~\cite{HR09} for the specific details.

In the notation of \cite[Theorem~3.4]{HR09}, we partition  $M$ as  $L_1 \cup L_2 \cup L_3 \cup L_4$ in a certain way and then $M \cup \{1^s\}$ is realizable for all $s \geq s_1 + s_2 + s_3 + s_4 -1$, where each~$s_i$ is dependent on~$L_i$ for $1 \leq i \leq 4$.  Let $U_i$ be the underlying set of $L_i$ for $1 \leq i \leq 4$.  

By \cite[Lemma~3.12]{HR09}, we may take $s_1 = 1 - 2|U_1| + \sum_{x \in U_1} x$; hence $s_1 \leq \left(\sum_{i=1}^{k} x_i \right) -1$.
By \cite[Lemma~3.9]{HR09}, we may take $s_2 = \max(U_2) - 1$; hence $s_2 \leq x_k -1$. 
By \cite[Lemma~3.7]{HR09}, we may take $s_3 = \max(U_3) - 1$; hence $s_3 \leq x_k -1$.
By \cite[Lemma~3.13]{HR09}, we may take $s_4 = \max(U_4) -  |U_4|$; hence $s_4 \leq x_k -1$.

Combining these bounds we find that $L$ is realizable for all $s \geq 3x_k - 5 + \sum_{i=1}^{k} x_i$.
\end{proof}

The BHR Conjecture has close connections to many other problems and conjectures concerning sequences with distinct partial sums or subgraphs of~$K_v$ other than paths; see~\cite{OPPS} for more discussion of this.  A recent paper also makes a connection between the BHR Conjecture and the Traveling Salesman Problem~\cite{GW20}.

We are frequently concerned with the congruence classes of multiple elements with respect to multiple integers.  The following notation is useful for these situations:  if $x_i \equiv y_i \pmod{z_i}$ for $1 \leq i \leq k$, then write
$$(x_1, \ldots, x_k) \equiv (y_1, \ldots, y_k) \pmod{(z_1, \ldots, z_k)}.$$ 
We also need the notion of a {\em translation} of a sequence~${\bm h} = [ h_1, \ldots, h_v ]$ by an integer~$m$:
$${\bm h} + m = [h_1+m, \ldots, h_v + m].$$
The translation of a sequence produces the same multiset of absolute differences as the original sequence.

We require a lot of small examples of realizations for particular multisets.  These were mostly found using a heuristic algorithm implemented in GAP.  
Given a target multiset~$L$ of size~$v-1$, the algorithm starts from a random Hamiltonian path in~$K_v$ and keeps trying to move to a Hamiltonian path that is ``closer" to realizing~$L$---in the sense of trying to increase $|L \cap L'|$, where~$L'$ is the multiset realized by the path under consideration---by removing an edge from the path and reconnecting the two resulting paths in a different way.  If it gets stuck before finding a realization of~$L$, then it tries again from a different starting path.  For the fairly small values of~$v$ in which we are interested, this simple algorithm is sufficient to find the desired realizations quickly.
The programs are available on the ArXiv page for this paper.

The most far-reaching components of Theorem~\ref{th:known} were proved using ``linear" realizations.    A Hamiltonian path ${\bm h} = [h_1, h_2, \ldots, h_{v}]$ of $K_v$ defines a multiset of absolute differences $L = \{ |h_i - h_{i+1}| : 1 \leq i \leq v-1 \}$ with underlying set contained in $\{1,\ldots, v-1\}$.  In this situation,  ${\bm h}$ is a {\em linear realization} of~$L$.  If $h_1 = 0$, then the linear realization is {\em standard}; if $h_1 = 0$ and $h_{v} = v-1$, then the linear realization is {\em perfect}.  To emphasize the distinction between linear realizations and realizations, realizations as defined above are sometimes referred to as {\em cyclic realizations}.

Linear realizations are closely related to cyclic realizations.  For example, if each element in a multiset~$L$ of size~$v-1$ is at most $\lfloor v/2 \rfloor$, then a linear realization of~$L$ is also a cyclic realization of~$L$.  See~\cite{HR09} for further discussion.

What makes linear realizations so useful in addressing the BHR Conjecture, standard and perfect ones especially, is their ability to be combined and hence used in inductive arguments.  This is the approach taken by Horak and Rosa in~\cite{HR09} and the multisets~$L$ given in Theorem~\ref{th:hr} in fact have linear realizations that are also cyclic realizations for the same~$L$.

The main contribution of this work is to introduce an alternative object: the ``growable" realization, which we define in the next section.  These are cyclic realizations that can be used in inductive arguments in somewhat similar ways to linear ones.  

In Section~\ref{sec:145} we reprove, with a much shorter proof, the result from~\cite{CD10} that $\BHR(L)$ holds when $L$ has underlying set~$U = \{1,2,3 \}$ to illustrate that this new tool is, in some ways, more powerful than existing ones.
We go on to prove instances of the BHR Conjecture that seem beyond the reach of current techniques. In particular, we are able to add the following items to Theorem~\ref{th:known}:
\begin{itemize}
\item $U = \{ 1,4,5 \}$ (Section~\ref{sec:145}), 
\item $U \subseteq \{ 1,2,3,4 \}$ (Section~\ref{sec:1234}),
\item $L = \{ 1^a, x^b, (2x)^c \}$ when $a \geq x-2$, $c$ is even and $b \geq 5x-2+c/2$ (Section~\ref{sec:1_x_2x}),
\item $L = \{ 1^a, 3^b, 6^c \}$ when $c$ is odd and $b \geq 18+(c-1)/2$ (Section~\ref{sec:1_x_2x}).
\end{itemize}

\section{Growable Realizations}\label{sec:grow}

Growable realizations will let us move from solving $\BHR(L)$ to $\BHR(L \cup \{ x^x \})$ under certain circumstances.  When this can be done for multiple choices of~$x$, this is a powerful tool.

Take $x$ with $0<x \leq v/2$.  For a given $m$, with $0 \leq m < v$, we shall embed $K_v$ into $K_{v + x}$ as follows:
$$
y \mapsto
\begin{cases}
y \hspace{2mm}  \text{when } y\leq m,  \\
y+x \hspace{2mm}  \text{otherwise.}
\end{cases}
$$
This embedding preserves some edge lengths and increases others.   

\begin{defn}
Let ${\bm h} = [h_1, \ldots, h_v]$ be a cyclic realization of a multiset~$L$.  Take $x$ and $m$ with $0<x \leq v/2$ and $0 \leq m < v$.  If 
each $y$ with $m-x < y \leq m$ is incident with exactly one edge whose length is increased by the embedding and there is no other edge whose length is increased, then say that ${\bm h}$ is {\em $x$-growable at $m$}.
\end{defn}

\begin{exa}\label{ex:picture}
It is easy to see that the sequence
$[6,4,3,0,7,1,5,2,8]$
is a cyclic realization of $\{ 1, 2^2, 3^4, 4 \}$ and that it is $3$-growable at $2$.
In fact, we can represent this Hamiltonian path of $K_{9}$ writing in bold the vertices not involved by the embedding 
$K_9\hookrightarrow K_{12}$, and using the symbol $-$ for each edge whose length does not change and  the symbol $\cdots$ for 
each edge whose length increases by $3$:
$$6- 4- 3 \cdots \bm{0} - 7  - \bm{1} \cdots 5 \cdots \bm{2} -8.$$
Note that every vertex in bold is incident with exactly one edge $\cdots$.
Also, note that the edges $0-7$, $7-1$ and $2-8$ do not change length, since their absolute differences are greater 
than $\left\lfloor \frac{9}{2}\right\rfloor$.
\end{exa}

Theorem~\ref{th:grow} and its immediate consequence Theorem~\ref{th:multigrow} are the core results for using growable realizations.

\begin{thm}\label{th:grow}
Suppose a multiset~$L$ has an $X$-growable realization. Then for each~$x \in X$, the multiset $L \cup \{ x^x \}$ has an  $X$-growable realization.
\end{thm}

\begin{proof}
Let ${\bm g} = [g_1, \ldots, g_v]$ be an $X$-growable realization of a multiset~$L$.  Take $x \in X$ and $m$ such that ${\bm g}$ is $x$-growable at $m$.  Each element $y$ with $m-x < y \leq m$ is adjacent to exactly one element~$z$ such that the edge between them is lengthened by the embedding of $K_v$ into $K_{v+x}$.   

Applying the embedding we obtain a sequence ${\bm h'} =  [h_1, \ldots, h_v]$ in $K_{v+x}$. 
Each adjacent pair $y,z$ in ${\bm g}$ as above becomes a subsequence $(y, z+x)$ or $(z+x,y)$ in~${\bm h'}$.  Obtain a new sequence~${\bm h}$ in $K_{v+x}$ by replacing each subsequence $(y,z+x)$  with $(y, y+x, z+x)$ and each  subsequence $(z+x,y)$  with $(z+x, y+x, y)$.  As there is one pair for each $y$ in the range $m-x < y \leq m$, this adds the elements $m+1, \ldots, m+x$ to the sequence and hence ${\bm h}$ is a Hamiltonian path in $K_{v+x}$.

Now, ${\bm h}$ has the desired lengths because each pair of adjacent elements in ${\bm g}$ whose length was fixed by the embedding are still adjacent in ${\bm h}$ and each adjacent pair $y,z$ whose length was not fixed is replaced by a triple whose lengths are the original length and $x$.  There are~$x$ such pairs.

If ${\bm g}$ is $x'$-growable at $m'$, then ${\bm h}$ is $x'$-growable at $m'$ if $m' \leq m$ and $x'$-growable at $m' + x$ if $m' > m$.
\end{proof}

\begin{exa}\label{ex:pic_ctd}
Applying the embedding  $K_9 \hookrightarrow K_{12}$ to the $3$-growable realization of Example~\ref{ex:picture} we obtain the sequence 
$$9- 7 - 6 \cdots 0 - 10 - 1 \cdots 8 \cdots 2 - 11.$$
Now, following the proof of Theorem~\ref{th:grow}, we insert the vertices $3,4,5$, replacing the edges 
$6 \cdots 0 $, $1 \cdots 8$ and $8 \cdots 2 $ with
$6 -3 \cdots 0 $, $1 \cdots 4-8$ and $8-5 \cdots 2 $, respectively.
In this way, the sequence
$$ 9 - 7 - 6 - 3 \cdots 0 - 10 - 1\cdots 4 -8 - 5 \cdots 2 -11$$
is a cyclic realization of  $\{1,2^2,3^7,4\}$, which is still $3$-growable at $2$.
\end{exa}

If a realization is $x$-growable for each $x \in X$ for some set~$X$, then say that it is {\em $X$-growable}.  

\begin{thm}\label{th:multigrow}
Suppose a multiset~$L$ has a realization that is $\{ x_1, \ldots, x_{k} \}$-growable.  Then the multiset  $L \cup \{ x_1^{x_1\ell_1 } , x_2^{x_2\ell_2 }, \ldots,   x_k^{x_k \ell_k }  \}$ has a $\{ x_1, \ldots, x_{k} \}$-growable realization for any $\ell_1, \ell_2, \ldots, \ell_k \geq 0$.
\end{thm}

\begin{proof}
Repeatedly apply Theorem~\ref{th:grow}.
\end{proof}

\begin{exa}\label{ex:1234} 
The sequence
$$ [0, 3, 6, 2, 1, 13, 10, 11, 14, 12, 9, 8, 5, 4, 7 ]$$
is a cyclic realization of $L = \{ 1^4, 2, 3^8, 4 \}$.  It is $1$-growable at $8$ and $9$; it is $2$-growable at $3$; it is $3$-growable at $11$; and it is $4$-growable at~$5$.

If we apply Theorem~\ref{th:grow} four times with $x=2$ and then three times with $x=3$ 
we get the sequence
\begin{flushleft}
$ [ 0, 3, 5, 7, 9, 11, 14, 10, 8, 6, 4, 2, 1, 30, 27, 24, 21, 18, 19, $
\end{flushleft}
\begin{flushright}
$ 22, 25, 28, 31, 29, 26, 23, 20, 17, 16, 
  13, 12, 15 ] $,
\end{flushright}
which is a $\{1,2,3,4\}$-growable realization of $\{ 1^4, 2, 3^8, 4 \} \cup \{2^8, 3^9\} =   \{ 1^4, 2^9, 3^{17}, 4 \}$. 
\end{exa}

Any standard linear realization (and hence any perfect realization) is 1-growable at~0.  

Suppose we are investigating multisets that have underlying set~$U = \{x_1, \ldots, x_k \}$.  Using Theorem~\ref{th:multigrow}, a $U$-growable realization for a multiset $L = \{x_1^{a_1}, \ldots, x_k^{a_k} \}$ is sufficient to cover all multisets $M = \{x_1^{b_1}, \ldots, x_k^{b_k} \}$ with $b_i \geq a_i$ for each $i$ and $$(b_1, \ldots, b_k) \equiv (a_1, \ldots, a_k) \pmod{ (x_1, \ldots, x_k)}.$$
This means that the task frequently breaks naturally into considering $\prod_{i=1}^k x_i$ cases according to congruence modulo  $(x_1, \ldots, x_k)$.

We conclude this section with two lemmas that allow the expansion of the range of values for which realizations are growable.

\begin{lem}\label{lem:1grow}
Suppose $L$ has an $X$-growable realization with $1 \in X$ and $K$ has a $Y$-growable perfect linear realization.  Then $L \cup K$ has a $(X \cup Y)$-growable realization. 
\end{lem}

\begin{proof}
Suppose $|K| = k$ and let ${\bm g} = [ g_1, \ldots, g_{k+1} ]$ be a  $Y$-growable perfect linear realization of~$K$.   
  
Apply Theorem~\ref{th:grow} $k$ times with $x = 1$ to the $X$-growable realization of~$L$ to obtain an $X$-growable realization of $L \cup \{1^k \}$ with subsequence $m, m+1, \ldots, m+k$.  Replace this subsequence with ${\bm g} + m$ to obtain the desired $(X \cup Y)$-growable realization of $L \cup K$.
\end{proof}

It is possible to take~$Y$ to be the empty set in Lemma~\ref{lem:1grow}  to construct an $X$-growable realization for~$L \cup K$.

\begin{lem}\label{lem:evengrow}
Suppose $L$ has an $X$-growable realization with $2 \in X$.  Let $y$ and $z$ be even (possibly with $y=z$).
Then $L \cup \{ 1^{y+z-4} , y^{y+1}, z^{z+1} \}$ has an $(X \cup \{ y,z \} )$-growable realization. 
\end{lem}

\begin{proof}
Apply Theorem~\ref{th:grow} $y+z-1$ times with $x=2$ to the $X$-growable realization of~$L$ to obtain an $X$-growable realization of $L \cup \{2^{2(y+z-1)} \}$ with the following two subsequences: 
$$[m, m+2, \ldots, m+ 2y+2z -2] , [m-1, m+1, \ldots, m+ 2y+2z -3].$$ 
The sequence
$${\bm g} = [ 1, y+1, y+2, 2, 3, y+3, \ldots, y-1, 2y-1, 2y + z -1, 2y+2z -1]$$
uses the elements 
$$\{1, 2, \ldots, y-1, y+1, y+2, \ldots, 2y-1, 2y + z -1, 2y+2z -1\}$$ 
and has edge-lengths $\{ 1^{y-2}, y^{y-1}, z^2 \}$. 
The sequence
$${\bm h} = [ 0, y, 2y, 2y+z, 2y+z+1, 2y+1, 2y+2, 2y+z+2, \ldots, 2y+2z-2  ]$$
uses the elements 
$$\{0, y, 2y, 2y+1, \ldots, 2y + z -2, 2y+z, 2y+z+1, \ldots,  2y+2z -2\},$$ 
and has edge-lengths $\{ 1^{z-2}, y^2, z^{z-1} \}$.  The elements used by~${\bm g}$ and~${\bm h}$ together are exactly those used in  the two subsequences from the realization of $L \cup \{2^{2(y+z-1)} \}$.
Replace the two subsequences  with ${\bm g} + m-1$ and ${\bm h} + m-1$ respectively to obtain a realization of  $L \cup \{ 1^{y+z-4} , y^{y+1}, z^{z+1} \}$.  

It is $y$-growable at $m+y-1$ because each $t$ in the range $m-1 <  t < m+y-1$ is adjacent to $t + y > m+y-1$ and to $t \pm 1 \leq m+y-1$, and $m+y-1$ is adjacent to $m-1$ and $m+2y-1$.  It is  $z$-growable at $m +2y + z - 2$ because each $t$ in the range $m + 2y - 2 <  t < m+2y + z-2$ is adjacent to $t + z > m +2y + z - 2$ and to $t \pm 1 \leq m +2y + z - 2$, and $m +2y + z - 2$ is adjacent to $m +2y - 2$ and $m +2y + 2z - 2$.
\end{proof}

\section{Complete Solutions for $U=\{1,2,3\}$ and $U = \{1,4,5 \}$}\label{sec:145}

Given any fixed set~$U$, we may use growable realizations to try to prove $\BHR(L)$ for all but finitely many multisets~$L$ with underlying set~$U$.  To do this, divide the problem into~$\prod_{x \in U} x$ cases, corresponding to the possible congruence classes of the number of occurrences of each element~$x \pmod{x}$.  For each case, a finite number---possibly one---of growable realizations can show that all but finitely many---possibly zero---admissible~$L$ matching these congruence classes has a realization.  The finitely many exceptions can then be dealt with directly.   In this section we illustrate this process for~$U = \{ 1,2,3 \}$ and $U = \{1,4,5\}$.

When $U = \{ 1,2,3 \}$, the BHR Conjecture is already known to hold~\cite{CD10}.   However, the self-contained proof given here in Theorem~\ref{th:123} is significantly shorter, which gives an indication of the power of the method of growable realizations compared to existing tools.

When $U = \{ 1,4,5 \}$, from previous work we know that $\{ 1^a, 4^b, 5^c \}$ is realizable when $a\geq 11$ or when both $a \geq 7$ and $b \geq c$~\cite{OPPS}.  However, the proof of Theorem~\ref{th:145} does not rely on this result.  

\begin{thm}\label{th:123}
Let  $L = \{1^a,2^b,3^c\}$ be an admissible multiset with $a,b,c \geq 1$. Then $\BHR(L)$ holds.
\end{thm}

\begin{proof}
We start with the $\{1,2,3\}$-growable cyclic realizations
of $\{1, 2^b, 3^c\}$  described in the first part of Table \ref{T123-1}, which
allow to cover all the $6$ possibilities of the congruence class combinations of $(b,c) \pmod{(2,3)}$.
Using Theorem~\ref{th:multigrow} this proves $\BHR(L)$ for all $a,b\geq 1$ and $c\geq 5$.
To complete the case $b+c\geq 4$,  we use the $\{1,2\}$-growable realizations for
$(b,c) \in \{(1,4),(2,2), (3,1),(3,2), (3,3),(4,1)\}$ from the second part of Table~\ref{T123-1} and
the $1$-growable realization of $\{1,2,3^3\}$, described in Table \ref{T123-2}.

Now, the cases when $b+c< 4$ can be solved using the $1$-growable realizations of
$\{1^a,2^b,3^c\}$, described in Table \ref{T123-2}.
\end{proof}

\begin{table}[ht]
\caption{$\{1,2,3\}$-growable cyclic realizations for $\{1,2^b,3^c\}$:
they are $x$-growable at $m_x$. The congruence classes of $(b,c)$ are taken modulo $(2,3)$.}\label{T123-1}
\begin{center}
\begin{footnotesize}
$$\begin{array}{llll}\hline
\text{Classes} & \text{Realizations} & (b,c) & (m_1,m_2, m_3) \\ \hline
(0,0) &  [ 2, 4, 1, 5, 3, 0, 6 ] & ( 2, 3) &  ( 5, 1, 3 ) \\
(0,1) &  [ 3, 6, 0, 5, 2, 1, 7, 4 ] & ( 2, 4) &  ( 2,3 , 4) \\
(0,2) &  [ 6, 5, 2, 8, 1, 4, 7, 0, 3 ] & ( 2, 5) & ( 7, 5, 2) \\
(1,0) &  [ 8, 5, 2, 3, 6, 0, 7, 1, 4 ] & (1, 6) & ( 1, 6, 3) \\
(1,1) &  [ 5, 8, 1, 4, 6, 9, 2, 3, 0, 7 ] & ( 1, 7) &  ( 4,  6, 2 ) \\
(1,2) &  [ 6, 1, 4, 7, 5, 0, 3, 2 ] & ( 1, 5) & ( 4, 1, 2) \\ \hline
(0,1) &  [ 0, 2, 4, 1, 6, 5, 3 ] & ( 4, 1) & ( 5,2, 3) \\
(0,2) &  [ 3, 1, 4, 5, 2, 0 ] & ( 2, 2) & ( 4,1,-) \\
(1,0) &  [ 7, 4, 2, 0, 3, 1, 6, 5 ] & ( 3, 3) & ( 4, 5,2) \\
(1,1) &  [ 2, 5, 1, 3, 6, 0, 4 ] & ( 1, 4) & ( 1, 3,-) \\
      &  [ 4, 2, 5, 3, 1, 0 ] & ( 3, 1) & ( 1,3, -) \\
(1,2) & [ 2, 4, 6, 5, 1, 3, 0 ] & ( 3, 2) & ( 4, 1,2)\\ \hline
\end{array}$$
\end{footnotesize}
\end{center}
\end{table}

\begin{table}[ht]
\caption{$1$-growable cyclic realizations for $\{1^a,2^b,3^c\}$:
they are $1$-growable at $m_1$.}\label{T123-2}
\begin{center}
\begin{footnotesize}
$$\begin{array}{lll|lll}\hline
(a,b,c) & \text{Realizations}  & m_1  & (a,b,c) & \text{Realizations}  & m_1  \\ \hline
(1,1,3) &  [ 2, 5, 4, 1, 3, 0 ] & 4 &
(2,1,2) &  [ 0, 3, 5, 4, 1, 2 ] & 3 \\
(2,2,1) &  [ 3, 1, 0, 5, 2, 4 ] & 1 &
(3,1,1) &  [ 0, 5, 4, 1, 3, 2 ] & 4 \\
\hline
\end{array}$$
\end{footnotesize}
\end{center}
\end{table}

We now move on to $U = \{ 1,4,5 \}$.

\begin{lem}\label{145-a2}
Let $L=\{1^a,4^b,5^c\}$ be an admissible multiset with $a\geq 2$.
Then $\BHR(L)$ holds.
\end{lem}

\begin{proof}
In view of Theorem~\ref{th:known}.1,  we may assume $b,c \geq 1$.
We start with the $\{1,4,5\}$-growable cyclic realizations of $\{1^2, 4^b, 5^c\}$  described in the first part of Table \ref{T145-1}
(note that in this case $b+c\geq 7$).
These realizations allow to cover all the $20$ possibilities of the congruence class combinations of $(b,c) \pmod{(4,5)}$.
Using Theorem~\ref{th:multigrow}, this proves $\BHR(L)$ for all $a\geq 2$, $b\geq 7$ and $c\geq 1$.
The case  $2\leq b\leq 6$ with $b+c\geq 8$ can be solved using
the $\{1,5\}$-growable cyclic realizations of $\{1^2,4^b,5^c\}$ provided by Table \ref{T145-1},
with the exception of $(b,c)\equiv (2,4)\pmod{(4,5)}$.
Furthermore, the same table gives $5$-growable cyclic realizations of $\{1^2, 4, 5^c\}$
for $c\geq 7$ with $c\not \equiv 1 \pmod 5$. Note that the multisets
$\{1^2, 4, 5^{5k+6}\}$ are not admissible.

To complete the case $b+c\geq 7$ we consider the $5$-growable cyclic realization
of $\{1^2,4^2,5^9\}$ and the $1$-growable cyclic realizations of $\{1^2, 4^b, 5^{7-b}\}$, $2\leq b\leq 6$,
given in Table \ref{T145-1}, as well as the $\{1,5\}$-growable cyclic realization
of $\{1^3,4^2,5^9\}$ given in Table \ref{T145-2}.

To conclude our proof,  we use the $1$-growable cyclic realizations of $\{1^a,4^b,5^c\}$ with $a+b+c=9$, described in Table
\ref{T145-2}.
\end{proof}

\begin{table}[ht]
\caption{$\{1,5\}$-growable cyclic realizations for $\{1^a,  4^b,5^c\}$, $a\geq 3$:
they are $x$-growable at $m_x$.}\label{T145-2}
\begin{center}
\begin{footnotesize}
$$\begin{array}{llll}\hline
(a,b,c) & \text{Realizations}  & (m_1, m_5) \\ \hline
(3,1,6 )& [ 6, 7, 2, 1, 5, 0, 10, 4, 9, 3, 8 ] & (9,4) \\
(3,2,9) & [ 9, 14, 0, 10, 5, 4, 8, 13, 3, 7, 12, 2, 1, 11, 6 ] & ( 3, 9) \\\hline
(3,1,5) &  [ 8, 3, 2, 7, 1, 6, 5, 0, 9, 4 ] & (2,-)   \\
(3,2,4) &  [ 7, 2, 6, 1, 5, 0, 9, 8, 3, 4 ] & (1,-)   \\
(3,3,3) &  [ 3, 2, 8, 4, 9, 0, 5, 1, 6, 7 ] & ( 8,-) \\
( 3,4,2) & [ 6, 2, 8, 7, 3, 4, 9, 0, 5, 1 ] & (7,-)    \\
(3,5,1) & [ 5, 6, 2, 8, 9, 4, 0, 1, 7, 3 ] & (7,-) \\
(4 ,1,4) & [ 2, 1, 6, 7, 3, 8, 9, 4, 5, 0 ] & (1,-)    \\
( 4,2,3) & [ 7, 8, 4, 9, 3, 2, 1, 6, 5, 0 ] & (4,-)  \\
(4,3,2) &  [ 9, 5, 0, 6, 1, 2, 3, 4, 8, 7 ] & (4,-) \\
(4,4,1) & [ 0, 9, 4, 3, 7, 8, 2, 6, 5, 1 ] & (8,-) \\
(5,1,3) & [ 8, 3, 2, 7, 6, 5, 1, 0, 9, 4 ] & (6,-)\\
(5,2,2) & [ 5, 4, 9, 0, 1, 2, 8, 3, 7, 6 ] & (8,-) \\
(5,3,1) & [ 4, 5, 9, 8, 3, 7, 6, 2, 1, 0 ] & (2,-) \\
(6,1,2) & [ 3, 4, 8, 9, 0, 5, 6, 7, 2, 1 ] & (8,-) \\
(6,2,1) & [ 8, 4, 3, 2, 1, 7, 6, 5, 0, 9 ] & (4,-) \\
(7,1,1) & [ 3, 2, 1, 0, 4, 9, 8, 7, 6, 5 ] & (8,-) \\ \hline
\end{array}$$
\end{footnotesize}
\end{center}
\end{table}

\begin{table}[ht]
\caption{$\{1,4,5\}$-growable cyclic realizations for $\{1^2,  4^b,5^c\}$:
they are $x$-growable at $m_x$. The congruence classes of
$(b,c)$ are taken modulo $(4,5)$.}\label{T145-1}
\begin{center}
\begin{footnotesize}
$$\begin{array}{llll}\hline
\text{Classes} & \text{Realizations} & (b,c) & (m_1, m_4, m_5) \\ \hline
(0,0) & [ 5, 9, 1, 6, 7, 2, 10, 3, 8, 4, 11, 0 ] & (4,5) & (9,4,5) \\
(0,1) & [ 5, 9, 1, 6, 2, 10, 11, 3, 7, 8, 4, 0 ] & (8,1) & (9,3,5) \\
(0,2)  & [ 5, 6, 1, 10, 9, 0, 4, 8, 12, 3, 7, 2, 11 ] & (8,2) & (8,3,5) \\
(0,3)    & [ 1, 11, 12, 2, 7, 3, 13, 4, 8, 9, 5, 0, 10, 6 ] & (8,3) & (10,5,6) \\
(0,4) & [ 1, 6, 7, 2, 9, 5, 0, 10, 3, 8, 4 ] & (4,4) & (9,3,4) \\
(1,0)  & [ 7, 8, 3, 11, 2, 10, 6, 1, 5, 9, 4, 0, 12 ] & (5,5) & (10,6,7) \\
(1,1) & [ 10, 1, 6, 2, 11, 12, 3, 7, 8, 4, 0, 9, 5 ] & (9,1) & (9,3,5) \\
(1,2)  & [ 5, 9, 13, 3, 7, 8, 4, 0, 10, 1, 6, 2, 12, 11 ] & (9,2) & (9,3,5) \\
(1,3)  & [ 5, 9, 2, 7, 6, 1, 0, 4, 8, 3, 10 ] & (5,3) & (9,3,5) \\
(1,4)  & [ 0, 1, 5, 9, 4, 11, 3, 8, 7, 2, 10, 6 ] & (5,4) & (10,4,6) \\
(2,0) & [ 4, 9, 13, 0, 10, 5, 1, 11, 6, 2, 3, 8, 12, 7 ] & (6,5) & (1,4,7) \\
(2,1)    & [ 7, 11, 1, 5, 9, 10, 6, 2, 12, 13, 3, 8, 4, 0 ] & (10,1) & (11,3,7) \\
(2,2) & [ 1, 6, 2, 9, 5, 0, 10, 3, 7, 8, 4 ] & (6,2) & (9,3,5) \\
(2,3) & [ 5, 9, 1, 6, 2, 10, 3, 7, 8, 4, 11, 0 ] & (6,3) & (9,4,5) \\
(2,4) & [ 5, 6, 1, 10, 9, 0, 8, 4, 12, 3, 7, 2, 11 ] & (6,4) & (8,4,5) \\
(3,0)  & [ 12, 13, 2, 6, 10, 0, 11, 1, 5, 9, 4, 14, 3, 8, 7 ] & (7,5) & ( 10, 6, 7 ) \\
(3,1) & [ 10, 3, 7, 8, 4, 0, 1, 6, 2, 9, 5 ] & (7,1) & (9,3,5) \\
(3,2)  & [ 11, 3, 7, 2, 10, 6, 1, 0, 4, 8, 9, 5 ] & (7,2) & (10,3,5) \\
(3,3)    & [ 11, 12, 3, 8, 4, 0, 9, 5, 1, 10, 2, 7, 6 ] & (7,3) & (9,3,6) \\
(3,4)  & [ 11, 12, 2, 7, 6, 1, 10, 0, 4, 8, 3, 13, 9, 5 ] & (7,4) & (9,3,5) \\ \hline
(0,1) & [ 12, 11, 3, 8, 4, 0, 5, 9, 1, 10, 2, 7, 6 ] & (4,6) & (9,5,6) \\
(0,2) & [ 3, 13, 4, 9, 10, 5, 0, 1, 6, 11, 7, 2, 12, 8 ] & (4,7) & (12,7,8) \\
(0,3) & [ 12, 2, 6, 1, 11, 7, 3, 13, 8, 9, 14, 10, 0, 5, 4 ] & (4,8) & (7,3,4) \\
(1,1) & [ 0, 5, 9, 8, 4, 13, 12, 3, 7, 2, 11, 1, 10, 6 ] & (5,6) & (10,5,6) \\
(1,2) & [ 4, 9, 5, 1, 12, 7, 2, 3, 8, 13, 14, 10, 0, 11, 6 ]  & (5,7) &  ( 1, 5, 8) \\
(2,0) & [ 10, 0, 5, 4, 14, 9, 8, 13, 3, 7, 12, 2, 6, 1, 11 ] & (2,10) & (7,10,4) \\
(2,1) & [ 2, 7, 0, 6, 1, 8, 3, 4, 9, 10, 5 ] & (2,6) & (1,4,5) \\
(2,2) & [ 5, 10, 11, 6, 1, 9, 4, 3, 8, 0, 7, 2 ] & (2,7) &  (1,4,5) \\
(2,3) & [ 5, 10, 1, 6, 11, 12, 7, 2, 3, 8, 0, 9, 4 ] & (2,8) & (1,4,5) \\
(3,0) &  [ 1, 6, 2, 8, 9, 3, 7, 0, 5, 4, 10 ] & (3,5)  & ( 7, -,  4 ) \\
(3,1) & [ 2, 7, 0, 8, 3, 4, 9, 1, 5, 10, 11, 6 ] & (3,6) & (1,4,6) \\
(3,2) & [ 10, 2, 7, 3, 11, 12, 4, 8, 0, 9, 1, 6, 5 ] & (3,7) & (8,4,5) \\
(3,3) &  [ 4, 9, 0, 1, 10, 5, 6, 11, 2, 12, 7, 3, 13, 8 ] & (3,8) & (3,6,8) \\
(3,4) & [ 0, 5, 10, 6, 1, 11, 7, 2, 12, 13, 3, 14, 4, 9, 8 ] & (3,9) & (11,7,8) \\ \hline
(1,0) & [ 11, 2, 7, 12, 3, 8, 13, 4, 9, 10, 6, 1, 0, 5 ] & (1,10) & (-,9,4) \\
(1,2) & [ 8, 3, 2, 7, 1, 6, 0, 10, 4, 9, 5 ] & (1,7) & (-,4,5) \\
(1,3) & [ 6, 11, 4, 9, 8, 1, 2, 7, 3, 10, 5, 0 ] & (1,8) & (-,5,6) \\
(1,4) & [ 5, 10, 2, 7, 8, 3, 11, 6, 1, 0, 9, 4, 12 ] & (1,9) & (-,4,5) \\\hline
(2,4) & [ 6, 11, 2, 7, 12, 3, 8, 4, 5, 10, 1, 0, 9, 13 ] & (2,9) & (-,5,6)  \\ \hline
(0,3)& [ 5, 0, 6, 1, 7, 2, 3, 9, 8, 4 ] & (4,3) & (4,-,-) \\
(1,2) & [ 4, 8, 3, 9, 5, 0, 6, 7, 1, 2 ] & (5,2) &  ( 4,-,-) \\
(2,0) & [ 9, 4, 0, 5, 6, 1, 7, 2, 3, 8 ] & (2, 5) & ( 7,-,-) \\
(2,1) & [ 9, 3, 4, 0, 6, 5, 1, 7, 2, 8 ] & (6,1) & ( 6,-,-) \\
(3,4)& [ 9, 4, 5, 0, 1, 6, 2, 8, 3, 7 ] & (3,4) & ( 8,-,-) \\\hline
\end{array}$$
\end{footnotesize}
\end{center}
\end{table}

\begin{thm}\label{th:145}
Let $L = \{1^a,4^b,5^c\}$ be an admissible multiset with $a,b,c \geq 0$.  Then $\BHR(L)$ holds.
\end{thm}

\begin{proof}
By Lemma \ref{145-a2} we are left with the case $L=\{1,4^b,5^c\}$ with $b,c \geq 1$.
The multiset $L$ is admissible only if $b+c\geq 8$. Also, the following multisets are not
admissible: $\{1,4,5^{5k+7}\}$, $\{1,4^2,5^{5k+6}\}$ and $\{1,4^{4k+1},5\}$.
The $\{4,5\}$-growable cyclic realizations of $\{1, 4^b, 5^c\}$  described in the first part of Table \ref{T145-3}
allow to cover all the $20$ possibilities of the congruence class combinations of $(b,c) \pmod{(4,5)}$.
Using Theorem~\ref{th:multigrow}, this proves $\BHR(L)$ for all $b\geq 2$ and $c\geq 6$.
To complete the case $b=1$ we use the $5$-growable cyclic realization of $\{1,4,5^{11}\}$ given in Table \ref{T145-3}.
Finally, the case $c\leq 5$ can be solved using the  $4$-growable cyclic realizations of Table \ref{T145-4},
as well as the  cyclic realization  $[ 0, 5, 9, 4, 8, 3, 7, 2, 1, 6 ]$ of $\{1,4^3,4^5\}$.
\end{proof}

\begin{table}[ht]
\caption{$\{4,5\}$-growable cyclic realizations for $\{1,  4^b,5^c\}$:
they are $x$-growable at $m_x$. The congruence classes of
$(b,c)$ are taken modulo $(4,5)$.}\label{T145-3}
\begin{center}
\begin{footnotesize}
$$\begin{array}{llll}\hline
\text{Classes} & \text{Realizations} & (b,c) & (m_4, m_5) \\ \hline
(0,0) & [ 4, 9, 5, 0, 1, 6, 10, 3, 7, 2, 8 ] & (4,5) & (3,4) \\
(0,1) & [ 9, 2, 7, 3, 10, 5, 1, 0, 8, 4, 11, 6 ] & (4,6) & (4,6) \\
(0,2) & [ 4, 9, 1, 5, 0, 8, 12, 11, 6, 10, 2, 7, 3 ] & (4,7) & (3,4) \\
(0,3) & [ 6, 10, 5, 0, 9, 13, 4, 8, 3, 12, 7, 2, 1, 11 ] & (4,8) & (5,6) \\
(0,4) & [ 3, 14, 4, 9, 5, 0, 10, 6, 1, 11, 12, 7, 2, 13, 8 ] & (4,9) & (7,8) \\
(1,0) & [ 6, 11, 3, 8, 0, 12, 7, 2, 10, 5, 1, 9, 4 ] & (1,10) & (4,6) \\
(1,1)  &[ 4, 9, 5, 0, 8, 3, 12, 7, 11, 10, 1, 6, 2 ] & (5,6) & (3,4) \\
(1,2) &  [ 12, 3, 7, 11, 2, 1, 6, 10, 0, 5, 9, 4, 13, 8 ] &    (5,7) & (6,8) \\
(1,3) & [ 4, 9, 10, 3, 8, 2, 7, 1, 6, 0, 5 ] & (1,8) & (3,4) \\
(1,4) & [ 10, 3, 8, 1, 2, 7, 0, 5, 9, 4, 11, 6 ] & (1,9) & (5,6) \\
(2,0) & [ 7, 2, 11, 6, 1, 10, 5, 0, 9, 13, 12, 3, 8, 4 ] & (2,10) & (3,4) \\
(2,1) & [ 7, 11, 2, 12, 3, 8, 4, 13, 9, 5, 0, 1, 10, 6 ] & (6,6) & (6,7)  \\
(2,2) & [ 1, 6, 0, 5, 10, 3, 7, 2, 8, 9, 4 ] & (2,7) & (3,4) \\
(2,3) & [ 9, 1, 2, 7, 0, 5, 10, 3, 8, 4, 11, 6 ] & (2,8) & (5,6) \\
(2,4) & [ 9, 4, 12, 3, 8, 0, 1, 6, 11, 7, 2, 10, 5 ] & (2,9) & (4,5) \\
(3,0) & [ 4, 9, 14, 3, 8, 13, 2, 7, 12, 1, 6, 11, 10, 0, 5 ] & (3,10) & (3,4) \\
(3,1) & [ 4, 9, 3, 8, 2, 6, 10, 0, 5, 1, 7 ] & (3,6) & (3,4) \\
(3,2) & [ 1, 6, 10, 3, 8, 4, 11, 0, 7, 2, 9, 5 ] & (3,7) & (4,5) \\
(3,3) & [ 10, 5, 0, 4, 9, 1, 6, 11, 2, 3, 8, 12, 7 ] & (3,8) &  ( 6, 7 )\\
(3,4) & [ 11, 6, 1, 2, 7, 12, 3, 8, 4, 13, 9, 0, 10, 5 ] & (3,9) & (4,5) \\\hline
(1,1) & [ 1, 6, 10, 11, 2, 7, 12, 3, 8, 13, 4, 9, 0, 5 ] & (1,11) & (9,4) \\ \hline
\end{array}$$
\end{footnotesize}
\end{center}
\end{table}

\begin{table}[ht]
\caption{$4$-growable cyclic realizations for $\{1,  4^b,5^c\}$:
they are $4$-growable at $m_4$.}\label{T145-4}
\begin{center}
\begin{footnotesize}
$$\begin{array}{lll}\hline
(b,c) & \text{Realizations} & m_4 \\ \hline
(4,4) & [ 7, 3, 8, 4, 9, 0, 5, 1, 6, 2 ] &  4 \\
(5,3) & [ 9, 4, 8, 3, 7, 2, 6, 0, 1, 5 ] &  4 \\
(5,4) & [ 9, 10, 3, 7, 1, 5, 0, 6, 2, 8, 4 ] & 3 \\
(5,5)  & [ 4, 8, 9, 1, 6, 11, 3, 7, 2, 10, 5, 0 ] & 3 \\
(6,2) & [ 7, 3, 8, 4, 0, 9, 5, 1, 6, 2 ] &  5 \\
(6,3) & [ 2, 6, 1, 7, 0, 4, 8, 3, 10, 9, 5 ] & 3 \\
(6,4)  & [ 5, 6, 1, 9, 4, 0, 8, 3, 11, 7, 2, 10 ] & 3 \\
(6,5)  & [ 5, 10, 6, 1, 9, 0, 4, 8, 12, 11, 3, 7, 2 ] & 4 \\
(7,1) & [ 7, 3, 9, 4, 8, 2, 6, 0, 1, 5 ] &  4\\
(7,2)  & [ 8, 4, 0, 10, 3, 7, 1, 6, 2, 9, 5 ]  & 4 \\
(7,3) &  [ 4, 8, 0, 1, 5, 9, 2, 6, 11, 7, 3, 10 ] & 7\\
(7,4)  & [ 8, 3, 12, 0, 4, 9, 5, 1, 10, 2, 6, 11, 7 ]  & 6 \\
(7,5) & [ 5, 10, 0, 9, 13, 4, 8, 7, 3, 12, 2, 6, 1, 11 ]  & 4 \\
(8,1)  & [ 7, 3, 10, 0, 4, 8, 1, 6, 2, 9, 5 ] &  3 \\
(8,2) & [ 4, 8, 0, 5, 9, 1, 6, 2, 10, 11, 3, 7 ]  & 3 \\
(8,3) & [ 6, 10, 11, 2, 7, 3, 12, 8, 4, 0, 5, 9, 1 ]  & 5 \\
(9,2)    & [ 4, 9, 8, 12, 3, 7, 11, 2, 6, 10, 1, 5, 0 ] & 3 \\
(10,1) & [ 6, 10, 1, 5, 9, 0, 4, 8, 12, 11, 2, 7, 3 ] &   3 \\   \hline
\end{array}$$
\end{footnotesize}
\end{center}
\end{table}

\section{A Complete Solution for $U \subseteq \{1,2,3,4 \}$}\label{sec:1234} 

In this section we prove $\BHR\left(\{1^a,2^b,3^c,4^d\}\right)$.  In view of Theorem~\ref{th:known}.2 and~\ref{th:known}.3, we may assume $c,d\geq 1$.  Also, by Theorem~\ref{th:known}.4 we have as a starting point that $\BHR(L)$ holds for $a\geq 3$ and also for $a=2$ when $b\geq 1$.  We begin by closing the case $a=2$.

\begin{lem}\label{1^234}
Let $L=\{1^2, 3^c,4^d\}$ be an admissible multiset with $c,d\geq 1$.
Then $\BHR(L)$ holds.
\end{lem}

\begin{proof}
First, note that $L$ is admissible only if $c+d\geq 5$.
The first part of Table~\ref{Tab1^234} collects  $\{3,4\}$-growable cyclic realizations for $L$
in  each of the $12$ possibilities of congruence class combinations of $(c,d)\pmod {(3,4)}$.
Using Theorem~\ref{th:multigrow},  this proves $\BHR(L)$  except in the following cases:  $d=1,2$; $d=3$ and $c\not\equiv 0\pmod{3}$;
$d=4$ and $c\equiv 1 \pmod{3}$.
So, we prove the validity of $\BHR(L)$ for these exceptional cases using
the $3$-growable cyclic realizations
for the cases $(c,d)\in \{(4,1),(5,1),(6,1),(3,2),(4,2),(5,2),(2,3),(4,3),(4,4)\}$
given in Table \ref{Tab1^234}.
The last case left open is $L=\{1^2,3,4^4 \}$, for which we take the following cyclic realization: $[ 0, 4, 5, 1, 2, 6, 3, 7 ]$.
\end{proof}

\begin{table}[tp]
\caption{$\{3,4\}$-growable cyclic realizations for $\{1^2,  3^c,4^d\}$:
they are $x$-growable at $m_x$. The congruence classes of
$(c,d)$ are taken modulo $(3,4)$.}\label{Tab1^234}
\begin{center}
\begin{footnotesize}
\begin{tabular}{lllll}\hline
Classes & Realizations & $(c,d)$ & $(m_3,m_4)$ & Missing cases\\  \hline
$(0,0)$ & $[ 3, 7, 1, 4, 0, 9, 2, 6, 5, 8 ]$ & $(3, 4)$ & $(2,3)$  & \\
$(0,1)$ & $[ 3, 7, 10, 2, 6, 5, 1, 9, 8, 4, 0 ]$ & $( 3, 5) $ & $(2,5)$ & $d=1$ \\
$(0,2)$ & $[ 6, 9, 10, 2, 1, 5, 8, 0, 4, 7, 3, 11 ]$ & $( 3, 6) $ & $(5,6)$ &  $d=2$ \\
$(0,3)$ & $[ 1, 5, 6, 2, 8, 0, 3, 7, 4 ]$ & $( 3, 3)$ & $(3,4)$ &  \\
$(1,0)$ & $[ 3, 7, 11, 10, 6, 2, 1, 5, 9, 0, 4, 8 ]$ & $( 1, 8) $ & $(2,7)$ & $d=4$ \\
$(1,1)$ & $[ 3, 6, 2, 7, 8, 4, 0, 1, 5 ]$ & $(1, 5)$ & $(2,4)$  & $d=1$ \\
$(1,2)$ & $[ 5, 9, 8, 4, 1, 7, 3, 2, 6, 0 ]$ & $(1, 6)$ & $(4,5)$ & $d=2$ \\
$(1,3)$ & $[ 4, 8, 7, 3, 0, 1, 5, 9, 2, 6, 10 ]$ & $(1, 7 )$ & $(3,6)$ & $d=3$ \\
$(2,0)$ & $[ 4, 8, 0, 3, 7, 6, 1, 5, 2 ]$ & $( 2, 4 )$ & $(3,4)$ &  \\
$(2,1)$ & $[ 5, 9, 3, 6, 2, 8, 7, 4, 0, 1 ]$ & $(2,5)$ & $(4,5)$ & $d=1$\\
$(2,2)$ & $[ 4, 0, 10, 6, 9, 2, 5, 1, 8, 7, 3 ]$ & $(2, 6)$ & $(2,3)$ & $d=2$  \\
$(2,3)$ & $[ 8, 0, 4, 3, 11, 7, 6, 9, 1, 5, 2, 10 ]$ & $(2, 7) $ & $(8,3)$ & $d=3$ \\\hline
$(0,1)$ & $[ 5, 8, 9, 2, 6, 3, 0, 1, 4, 7 ]$ & $( 6, 1)$ & $(4,5)$ \\
$(0,2)$ & $[ 6, 3, 2, 5, 1, 4, 0, 7 ]$ & $( 3, 2) $ & $(3,-)$ \\
$(1,0)$ & $[ 8, 5, 6, 2, 10, 9, 1, 4, 0, 7, 3 ]$ & $(4,4)$ & $(2,3)$ & $(c,d)=(1,4)$ \\
$(1,1)$ & $[ 2, 5, 6, 3, 7, 4, 1, 0 ]$ & $( 4, 1 ) $ & $(4,-)$ \\
$(1,2)$ & $[ 3, 7, 8, 5, 2, 1, 4, 0, 6 ]$ & $(4, 2 )$ & $(2,3)$ \\
$(1,3)$ & $[ 2, 6, 3, 0, 9, 5, 1, 8, 7, 4 ]$ & $(4,3)$  & $(3,5)$ \\
$(2,1)$ & $[ 0, 4, 1, 7, 8, 2, 5, 6, 3 ]$ & $( 5, 1)$ & $(2,3)$  \\
$(2,2)$ & $[ 5, 8, 1, 4, 7, 3, 2, 6, 9, 0 ]$ & $( 5, 2)$ & $(4,5)$ \\
$(2,3)$ & $[ 7, 6, 2, 3, 0, 4, 1, 5 ]$ & $(2, 3)$ & $(2,-)$ \\
\hline
\end{tabular}
\end{footnotesize}
\end{center}
\end{table}

\begin{lem}\label{134}
Let $L=\{1, 2^b, 3^c,4^d\}$ be an admissible multiset, where
$b\geq 0$ is even and $c,d\geq 1$.
Then $\BHR(L)$ holds.
\end{lem}

\begin{proof}
Suppose first $d\geq 5$. We start with the $\{2,3,4\}$-growable cyclic realizations of $\{1,3^c,4^d\}$ described in the first part of Table \ref{Tab234}
(note that in this case $c+d\geq 6$).
These realizations allow to cover all the $12$ possibilities of the congruence class combinations of $(c,d)\pmod {(3,4)}$.
Using Theorem~\ref{th:multigrow}  this proves $\BHR(L)$ except when $c=1$ and $d\not \equiv 2 \pmod{4}$.
So, suppose $c=1$.
Table \ref{Tab234} also gives $\{2,3,4\}$-growable cyclic realizations for $d=7,8$, proving the validity of $\BHR(L)$
for $b\geq 0$ even, $c=1$ and $d\equiv 0,3\pmod 4$. Hence, we may assume $d\equiv 1 \pmod 4$.
Note that the multiset $L=\{1,3,4^{4k+5}\}$ does not satisfy the necessary condition, as $4k+5 > v-4= 4k+4$.
On the other hand, a  $\{2,3,4\}$-growable cyclic realization of $\{1,2^2,3,4^5\}$ is given in
Table \ref{Tab234-2}. 
Next, we consider the cases $1\leq d\leq 4$.
Table \ref{Tab234} provides $\{2,3\}$-growable cyclic realizations for  $\{1, 3^c,4^d\}$,
in the following cases:
$d=1$ and $5\leq c \leq 7$; $d=2$ and $4\leq c \leq 6$; $d=3,4$ and $3\leq c \leq 5$.
It also provides a $2$-growable cyclic realization for the multiset $\{1,3^2,4^4\}$.
This completes the analysis for the admissible multisets with $b=0$.

Now, Table \ref{Tab234-2} gives
$2$-growable cyclic realizations for the multiset $\{1,2^2, 3^c,4^d\}$ when
$(c,d)\in \{(1,3), (1,4),(2,2), (2,3), (3,1),(3,2), (4,1)\}$, that is with $4\leq c+d \leq 5$.  
This completes the analysis for the admissible multisets with $b=2$.
Finally, Table \ref{Tab234-2} also gives $2$-growable cyclic realizations for the multiset $\{1,2^4, 3^c,4^d\}$ for each 
$(c,d)$ in the set  $\{(1,1),(1,2),(2,1)\}$, concluding our proof.
\end{proof}

\begin{table}[tp]
\caption{$\{2,3,4\}$-growable cyclic realizations for $\{1,3^c,4^d\}$: they are $x$-growable at $m_x$.
The congruence classes of $(c,d)$ are taken modulo $(3,4)$.}\label{Tab234}
\begin{center}
\begin{footnotesize}
\begin{tabular}{lllll}\hline
Classes & Realizations & $(c,d)$ & $(m_2,m_3,m_4)$ & Missing cases \\  \hline
$(0,0)$ & $[ 3, 6, 1, 4, 0, 5, 2, 7, 8 ]$ & $(3,4)$ & $(6,2,3)$ \\
$(0,1)$ & $[ 3, 6, 2, 8, 4, 1, 7, 0, 9, 5 ]$ &   $(3, 5)$  & $(2,4,5)$ & $d=1$ \\
$(0,2)$ & $[ 4, 5, 1, 8, 0, 3, 7, 10, 6, 2, 9 ]$ &  $(3, 6)$ & $(7,3,4)$  & $d=2$\\
$(0,3)$         & $[ 3, 7, 6, 2, 10, 1, 5, 9, 0, 4, 8, 11 ]$ & $(3,7)$ & $(9,2,6)$ & $d=3$ \\
$(1,0)$        & $[ 4, 7, 0, 6, 3, 9, 8, 2, 5, 1 ]$ & $(4, 4)$ & $(7,3,4)$ & $c=1$ \\
$(1,1)$ & $[ 6, 9, 2, 10, 3, 7, 4, 0, 1, 8, 5] $ &  $(4, 5)$ & $(4,5,6)$ & $c=1$ or $d=1$ \\
$(1,2)$ & $[ 2, 6, 1, 5, 0, 3, 7, 8, 4 ]$ & $(1, 6)$ & $(1,3,4)$ & $d=2$ \\
$(1,3)$        & $[ 4, 7, 1, 5, 0, 8, 2, 6, 3 ]$ & $(4, 3)$ & $(2,3,4)$ & $c=1$ \\
$(2,0)$        & $[ 7, 11, 3, 4, 0, 8, 5, 1, 9, 6, 2, 10 ]$ & $(2,8)$ & $(6,8,3)$ & $d=4$  \\
$(2,1)$ & $[ 3, 7, 2, 6, 0, 1, 5, 8, 4 ]$ &  $(2,5)$  & $(2,3,4)$ & $d=1$ \\
$(2,2)$ & $[ 4, 5, 1, 8, 2, 6, 0, 7, 3, 9 ]$ & $(2,6)$  & $(7,3,4)$ & $d=2$ \\
$(2,3)$ & $[ 4, 8, 1, 0, 7, 10, 3, 6, 2, 9, 5 ]$ & $(2,7)$ & $(3,4,5)$ & $d=3$ \\\hline
$(1,0)$ & $[ 4, 5, 1, 8, 0, 7, 3, 10, 6, 2, 9 ]$ &  $( 1, 8 )$  & $(7,3,4)$\\
$(1,3)$ & $[ 3, 7, 1, 4, 8, 2, 6, 0, 9, 5 ]$ &  $(1,7)$  & $(2,4,5)$ \\\hline
$(0,1)$         & $[ 7, 1, 4, 5, 8, 2, 6, 0, 3 ]$ & $(6,1)$ &  $(6,2,-)$ \\
$(0,2)$         & $[ 5, 6, 3, 9, 2, 8, 1, 4, 7, 0 ]$ & $(6,2)$ & $(4,5,-)$\\
$(0,3)$ & $[ 0, 3, 7, 6, 2, 5, 1, 4 ]$ & $(3,3)$ & $(1,2,-)$ \\
$(1,1)$        & $[ 5, 2, 9, 8, 1, 4, 7, 0, 6, 3 ]$  &  $(7,1)$ & $(7,4,-)$ \\
$(1,2)$        & $[ 0, 5, 2, 6, 1, 4, 3, 7 ]$ & $(4,2)$ &  $(2,3,-)$ \\
$(2,0)$        & $[ 10, 9, 2, 6, 3, 0, 7, 4, 1, 8, 5 ]$ & $(5,4)$  & $(8,4,5)$ & $(c,d)=(2,4)$ \\
$(2,1)$        & $[ 2, 5, 0, 1, 6, 3, 7, 4 ]$ & $(5,1)$ & $(3,4,-)$ \\
$(2,2)$        & $[ 2, 5, 8, 7, 3, 0, 6, 1, 4 ]$ & $(5,2)$ & $(1,3,-)$ \\
$(2,3)$        & $[ 3, 0, 6, 7, 1, 4, 8, 5, 2, 9 ]$ & $(5,3)$ & $(5,2,-)$ \\ \hline
$(2,0)$ & $[ 2, 6, 7, 3, 0, 4, 1, 5 ]$ & $( 2, 4 )$ & $(1,-,-)$ \\
\hline
\end{tabular}
\end{footnotesize}
\end{center}
\end{table}

\begin{table}[tp]
\caption{$\{2,3,4\}$-growable cyclic realizations for $\{1, 2^b, 3^c,4^d\}$, with $b\geq 2$ even:
they are $x$-growable at $m_x$. The congruence classes of $(c,d)$ are taken modulo $(3,4)$.}\label{Tab234-2}
\begin{center}
\begin{footnotesize}
\begin{tabular}{llll}\hline
Classes & Realizations & $(b,c,d)$ & $(m_2,m_3,m_4)$ \\  \hline
$(1,1)$ & $[ 7, 8, 2, 6, 0, 4, 1, 9, 5, 3 ]$ & $(2,1,5)$ & $(6,2,3)$ \\\hline
$(0,1)$ & $[ 6, 4, 1, 2, 5, 7, 3, 0 ]$ & $(2,3,1)$ & $(5,2,-)$ \\
$(0,2)$ & $[ 4, 7, 5, 1, 8, 0, 3, 6, 2 ]$ & $(2,3,2)$ & $(1,3,4)$ \\
$(1,0)$ & $[ 1, 5, 8, 6, 2, 7, 0, 4, 3 ]$ & $(2, 1, 4)$ & $(6,2,3)$ \\
$(1,1)$ & $[ 3, 6, 0, 4, 1, 8, 7, 5, 2 ]$ & $(2,4,1)$ & $(6,2,3)$ \\
$(1,3)$ & $[ 1, 0, 4, 6, 2, 5, 3, 7 ]$ & $(2, 1, 3)$ & $(3,4,-)$ \\
$(2,2)$ & $[ 7, 6, 2, 4, 1, 5, 3, 0 ]$ & $(2, 2, 2)$ & $(1,3,-)$ \\
$(2,3)$ & $[ 3, 5, 7, 8, 2, 6, 1, 4, 0 ]$ & $(2, 2, 3 )$ & $(6,2,3)$ \\\hline
$(1,1)$ & $[ 0, 3, 1, 7, 5, 4, 2, 6 ]$ & $(4, 1, 1)$ & $(1,2,-)$\\
$(1,2)$ & $[ 3, 5, 7, 8, 6, 2, 0, 4, 1 ]$ & $( 4, 1, 2)$ & $(6,2,3)$ \\
$(2,1)$ & $[ 1, 3, 5, 8, 2, 4, 0, 7, 6 ]$ & $(4, 2, 1 )$ & $(6,2,-)$ \\
\hline
\end{tabular}
\end{footnotesize}
\end{center}
\end{table}

\begin{lem}\label{1234}
Let $L=\{1, 2^b, 3^c,4^d\}$ be an admissible multiset, where  $b\geq 1$ is odd and $c,d\geq 1$. Then $\BHR(L)$ holds.
\end{lem}

\begin{proof}
The first part of Table \ref{Tab234-odd} gives  $\{2,3,4\}$-growable cyclic realizations for $\{1, 2, 3^c, 4^d\}$
for  each of the $12$ possibilities of congruence class combinations of $(c,d) \pmod{(3,4)}$. Note that $c+d\geq 5$.
Using Theorem~\ref{th:multigrow},  this proves $\BHR(L)$  except for the following cases:
$d=1,2$; $d=3$ and $c\not \equiv 0\pmod{3}$;  $d=4$ and $c\equiv 1 \pmod 3$.
So, we prove the validity of $\BHR(L)$ for these exceptional cases using
$\{2,3\}$-growable cyclic realizations for  $\{1,2, 3^c,4^d\}$ and a $2$-growable cyclic realization for  $\{1, 2, 3, 4^4\}$, which can be found in Table~\ref{Tab234-odd}.

To conclude the proof we have to consider the cases when $c+d\leq 4$.
Table \ref{Tab234-odd} also provides $2$-growable cyclic realizations for
$\{1,2^3, 3^c,4^d\}$ when $(c,d)$ is in the set $\{(1,2),(1,3), (2,1), (2,2), (3,1)\}$,
and for $\{1,2^5, 3,4\}$.
\end{proof}

\begin{table}[htp]
\caption{$\{2,3,4\}$-growable cyclic realizations for $\{1, 2^b, 3^c,4^d\}$, with $b\geq 1$ odd:
they are $x$-growable at $m_x$. The congruence classes of $(c,d)$ are taken modulo $(3,4)$.}\label{Tab234-odd}
\begin{center}
\begin{footnotesize}
\begin{tabular}{lllll}\hline
Classes & Realizations & $(b,c,d)$ & $(m_2,m_3,m_4)$ & Missing cases \\  \hline
$(0,0)$ & $[ 9, 2, 6, 0, 4, 1, 7, 8, 5, 3 ]$ & $(1, 3, 4)$ & $(6,2,3)$ \\
$(0,1)$ & $[ 5, 8, 1, 9, 10, 2, 6, 4, 0, 7, 3 ]$ & $( 1, 3, 5)$ & $(8,4,5)$ & $d=1$ \\
$(0,2)$ & $[ 3, 5, 9, 1, 4, 0, 8, 7, 10, 6, 2, 11 ]$ & $(1, 3, 6)$ & $(6,2,3)$ & $d=2$\\
$(0,3)$ & $[ 4, 7, 3, 0, 1, 5, 8, 6, 2 ]$ & $( 1, 3, 3 )$ & $(1,3,4)$ \\
$(1,0)$ & $[ 10, 6, 2, 11, 3, 7, 9, 1, 5, 4, 0, 8 ]$ & $(1, 1, 8) $ & $(7,8,4)$ & $d=4$ \\
$(1,1)$ & $[ 2, 6, 7, 3, 0, 5, 1, 8, 4 ]$ & $( 1, 1, 5)$ & $(1,3,4)$ & $d=1$ \\
$(1,2)$ & $[ 3, 4, 0, 6, 2, 8, 1, 5, 9, 7 ]$ & $( 1, 1, 6 )$ & $(6,2,3)$ &  $d=2$\\
$(1,3)$ & $[ 9, 2, 6, 5, 1, 10, 3, 7, 0, 8, 4 ]$ & $(1, 1, 7 )$ & $(8,3,5)$ &  $d=3$\\
$(2,0)$ & $[ 2, 6, 1, 4, 0, 8, 5, 7, 3 ]$ & $( 1, 2, 4)$ & $(1,2,3)$ \\
$(2,1)$ & $[ 8, 1, 5, 4, 0, 6, 2, 9, 7, 3 ]$ & $( 1, 2, 5 )$ & $(7,2,4)$ & $d=1$\\
$(2,2)$ & $[ 3, 7, 0, 4, 6, 10, 2, 5, 1, 8, 9 ]$ & $(1, 2, 6)$ & $(7,2,4)$ & $d=2$\\
$(2,3)$ & $[ 7, 11, 3, 4, 0, 9, 1, 5, 8, 10, 6, 2 ]$ & $(1, 2, 7)$ & $(6,8,3)$ & $d=3$\\\hline
$(0,1)$ & $[ 4, 7, 9, 2, 6, 3, 0, 1, 8, 5 ]$ & $(1, 6, 1)$ & $(3,4,5)$  \\
$(0,2)$ & $[ 2, 5, 1, 0, 6, 3, 7, 4 ]$ & $(1, 3, 2)$ & $(3,4,-)$ \\
$(1,0)$ & $[ 3, 7, 10, 8, 0, 4, 5, 1, 9, 6, 2 ]$ & $(1, 4, 4)$ & $(7,2,4)$ & $(c,d)=(1,4)$\\
$(1,1)$ & $[ 0, 6, 3, 7, 4, 1, 2, 5 ]$ & $(1, 4, 1)$ & $(2,4,-)$  \\
$(1,2)$ & $[ 3, 7, 6, 0, 4, 1, 8, 5, 2 ]$ & $(1, 4, 2)$ & $(1,2,3)$ \\
$(1,3)$ & $[ 3, 6, 9, 7, 0, 1, 5, 2, 8, 4 ]$ & $(1, 4, 3)$ & $(2,3,4)$  \\
$(2,1)$ & $[ 3, 6, 7, 1, 4, 0, 2, 5, 8 ]$ & $(1, 5, 1)$ & $(6,2,3)$ & \\
$(2,2)$ & $[ 7, 1, 4, 0, 2, 5, 8, 9, 6, 3 ]$ & $(1, 5, 2)$ & $(6,2,3)$  \\
$(2,3)$ & $[ 4, 0, 3, 1, 5, 2, 6, 7 ]$ & $(1, 2, 3)$ & $(1,2,-)$ \\ \hline
$(1,0)$ & $[ 3, 7, 6, 2, 0, 4, 1, 5 ]$ & $(1,1,4)$ & $(1,-,-)$ \\\hline
$(0,1)$ & $[ 3, 6, 8, 7, 5, 2, 0, 4, 1 ]$ & $( 3, 3, 1)$ & $(6,2,3)$ & \\
$(1,0)$ & $[ 4, 6, 0, 8, 7, 3, 9, 1, 5, 2 ]$ &  $(3, 1, 4)$ & $(7,3,4)$ \\
$(1,1)$ & $[ 4, 2, 0, 8, 6, 3, 1, 5, 7 ]$ & $( 5, 1, 1)$ & $(6,3,-)$ \\
$(1,2)$ & $[ 4, 6, 2, 5, 3, 7, 1, 0 ]$ & $(3, 1, 2)$ & $(3,4,-)$ \\
$(1,3)$ & $[ 2, 6, 7, 0, 3, 5, 1, 8, 4 ]$ & $( 3, 1, 3) $ &  $(1,3,4)$ \\
$(2,1)$ & $[ 5, 2, 0, 1, 7, 3, 6, 4 ]$ & $(3, 2, 1)$ & $(3,4,-)$ \\
$(2,2)$ & $[ 4, 6, 0, 2, 5, 1, 8, 7, 3 ]$ & $(3, 2, 2)$ & $(2,3,4)$ \\
\hline
\end{tabular}
\end{footnotesize}
\end{center}
\end{table}

\begin{lem}\label{234odd}
Let $L=\{2^b, 3^c,4^d\}$ be an admissible multiset, where  $b\geq 1$ is odd and $c,d\geq 1$. Then $\BHR(L)$ holds.
\end{lem}

\begin{proof}
The first part of Table \ref{Tab0234odd} gives  $\{2,3,4\}$-growable cyclic realizations for $\{2, 3^c, 4^d\}$
in  each of the $12$ possibilities of congruence class combinations of $(c,d)\pmod{(3,4)}$. Note that $c+d\geq 6$.
Using Theorem~\ref{th:multigrow},  this proves $\BHR(L)$  except for the following cases:
$d=1,2,3$; $d=4,5$ and $c \equiv 2 \pmod{3}$; $c=1$ and $d\equiv 0,1 \pmod{4}$.
Next, we consider the case $c\geq 2$ and $1\leq d\leq 5$,
using  $\{2,3\}$-growable cyclic realizations for  $\{2, 3^c,4^d\}$.
For the exceptional case $\{2,3^2,4^4\}$ we use a $2$-growable cyclic realization.
Now we complete the case  $c\geq 2$: Table \ref{Tab0234odd}  also provides $2$-growable cyclic
realization for the multisets $\{2^3, 3^c,4^d\}$ when $(c,d)\in \{(2,2),(2,3), (3,1),(3,2),(4,1)\}$, and for the multiset $\{2^5, 3^2,4\}$.

Finally, we assume $c=1$. Note that the multiset $\{2,3,4^{4k+5}\}$ does not satisfy the necessary condition.
For the multisets $\{2,3,4^{4k+8} \}$ we use the $4$-growable realization
$[ 2, 6, 10, 3, 7, 4, 0, 9, 5, 1, 8 ]$ of $\{2,3,4^8\}$.
Table \ref{Tab0234odd} gives $\{2,4\}$-growable cyclic realizations for the multisets $\{2^3, 3,4^4\}$ and $\{2^3, 3,4^5\}$;
it also gives $2$-growable cyclic realizations for the multisets $\{2^3,3,4^3\}$, $\{2^5,3,4\}$ and $\{2^5,3,4^2\}$.
\end{proof}

\begin{table}[ht]
\caption{$\{2,3,4\}$-growable cyclic realizations for $\{2^b, 3^c,4^d\}$, with $b\geq 1$ odd:
they are $x$-growable at $m_x$. The congruence classes of $(c,d)$ are taken  modulo $(3,4)$.}\label{Tab0234odd}
\begin{center}
\begin{footnotesize}
\begin{tabular}{lllll}\hline
Classes & Realizations & $(b,c,d)$ & $(m_2,m_3,m_4)$ & Missing cases\\  \hline
$(0,0)$ & $[ 2, 6, 8, 5, 0, 4, 1, 7, 3 ]$ & $(1, 3, 4 )$ & $(1,2,3)$ \\
$(0,1)$ & $[ 2, 5, 1, 8, 4, 0, 6, 9, 7, 3 ]$ & $( 1, 3, 5)$ & $(1,2,4)$ & $d=1$ \\
$(0,2)$ & $[ 3, 6, 10, 7, 0, 9, 2, 5, 1, 8, 4 ]$ & $(1, 3, 6)$ & $(2,3,4)$ & $d=2$ \\
$(0,3)$ & $[ 6, 10, 2, 5, 9, 0, 8, 4, 1, 3, 11, 7 ]$ & $(1, 3, 7)$ & $(5,6,7)$ & $d=3$  \\
$(1,0)$ & $ [ 3, 6, 9, 5, 2, 8, 0, 4, 1, 7 ]$ & $( 1, 4, 4)$  & $(6,2,3)$ & $c=1$ \\
$(1,1)$ & $[ 2, 5, 9, 6, 3, 10, 8, 1, 4, 0, 7 ]$ & $(1, 4, 5)$ & $(6,7,3)$ & $c=1$ or $d=1$ \\
$(1,2)$ & $[ 3, 7, 2, 6, 0, 5, 1, 8, 4 ]$ & $(1, 1, 6)$ & $(2,3,4)$ & $d=2$\\
$(1,3)$ & $[ 2, 6, 0, 4, 8, 1, 5, 9, 7, 3 ]$ & $(1, 1, 7)$ & $(1,2,5)$ & $d=3$ \\
$(2,0)$ & $[ 10, 2, 6, 3, 11, 7, 5, 1, 9, 0, 8, 4 ]$ & $(1, 2, 8)$  & $(8,3,5)$ & $d=4$\\
$(2,1)$ & $[ 7, 11, 8, 12, 3, 5, 9, 0, 4, 1, 10, 6, 2 ]$ & $(1, 2, 9)$ &  $(6,7,3)$ & $d=1,5$ \\
$(2,2)$ & $[ 2, 6, 0, 3, 7, 9, 5, 1, 8, 4 ]$ & $(1, 2, 6)$ &  $(1,3,5)$ & $d=2$ \\
$(2,3)$ & $[ 9, 2, 5, 1, 8, 0, 7, 3, 10, 6, 4 ]$ & $(1, 2, 7)$ & $(7,3,4)$ & $d=3$ \\\hline
$(0,1)$ & $[ 3, 6, 0, 4, 1, 7, 5, 2, 8 ]$ & $( 1, 6, 1)$ & $(6,2,3)$  \\
$(0,2)$ & $[ 4, 7, 1, 5, 2, 9, 6, 3, 0, 8 ]$ & $(1, 6, 2)$ & $(7,3,4)$  \\
$(0,3)$ & $[ 0, 4, 1, 7, 3, 6, 2, 5 ]$ & $(1, 3, 3)$ &  $ (2,3,-)$ & \\
$(1,1)$ & $[ 3, 6, 9, 2, 5, 1, 8, 0, 7, 4 ]$ & $( 1, 7, 1)$ & $(7,3,4)$  \\
$(1,2)$ & $[ 7, 1, 4, 0, 5, 2, 6, 3 ]$ & $(1, 4, 2)$ & $(2,3,-)$ \\
$(1,3)$ & $[ 3, 6, 0, 4, 2, 7, 1, 5, 8 ]$ & $( 1, 4, 3)$ & $(6,2,3)$  \\
$(2,0)$ & $[ 6, 10, 2, 9, 1, 4, 8, 0, 3, 7, 5 ]$ & $(1, 5, 4)$ &  $(4,5,6)$ & $(c,d)=(2,4)$\\
$(2,1)$ & $[ 2, 6, 1, 3, 7, 4, 0, 5, 8 ]$ & $(1, 2, 5)$ & $(1,4,-)$  \\
        & $[ 1, 4, 7, 3, 6, 0, 5, 2 ]$  & $(1, 5, 1)$ & $(2,4,-)$  \\
$(2,2)$ & $[ 3, 6, 0, 4, 1, 8, 5, 2, 7 ]$ & $( 1, 5, 2)$ & $(6,2,3)$  \\
$(2,3)$ & $[ 9, 6, 2, 8, 1, 5, 3, 0, 7, 4 ]$ & $ ( 1, 5, 3)$ & $(7,3,4)$  \\\hline
$(2,0)$ & $[ 7, 3, 6, 2, 4, 0, 5, 1 ]$  & $(1,2,4)$ & $(3,-,-)$\\ \hline
$(0,1)$ & $[ 1, 6, 0, 2, 5, 3, 7, 4 ]$ & $(3, 3, 1)$ & $(3,4,-)$ \\
$(0,2)$ & $[ 7, 1, 4, 0, 2, 6, 8, 5, 3 ]$ & $( 3, 3, 2)$ & $(6,2,3)$ \\
$(1,1)$ & $[ 3, 6, 8, 2, 5, 7, 0, 4, 1 ]$ & $(3, 4, 1)$ & $(6,2,3)$\\
$(2,1)$ & $[ 5, 7, 3, 1, 8, 2, 0, 6, 4 ]$ & $( 5, 2, 1)$ & $(4,5,-)$\\
$(2,2)$ & $[ 2, 4, 1, 5, 7, 3, 0, 6 ]$ & $( 3, 2, 2)$ & $(1,2,-)$ \\
$(2,3)$ & $[ 3, 5, 7, 0, 4, 1, 6, 2, 8 ]$ & $(3, 2, 3)$ & $(6,2,3)$  \\ \hline
$(1,0)$ &  $[ 3, 7, 5, 0, 4, 1, 8, 6, 2 ]$ & $( 3, 1, 4) $ & $(1,2,3)$ \\
$(1,1)$ & $[ 5, 7, 1, 3, 0, 6, 2, 4 ]$ & $(5, 1, 1)$ & $(1,2,-)$ \\
$(1,2)$ & $[ 4, 6, 8, 1, 5, 2, 0, 7, 3 ]$ & $(5, 1, 2)$ & $(2,3,4)$ \\
$(1,3)$ & $[ 1, 7, 3, 5, 2, 6, 4, 0 ]$ & $( 3, 1, 3)$ & $(3,4,-)$ \\
$(1,4)$ & $[ 5, 9, 1, 7, 3, 0, 8, 2, 6, 4 ]$ & $( 3, 1, 5) $ & $(3,4,5)$\\
\hline
\end{tabular}
\end{footnotesize}
\end{center}
\end{table}

\begin{thm}\label{th1234}
Let $L=\{1^a,2^b, 3^c,4^d\}$ be an admissible multiset with $a,b,c,d\geq 0$.
Then $\BHR(L)$ holds.
\end{thm}

\begin{proof}
By Theorem~\ref{th:known}.2--4 we may assume $0\leq a\leq 2$ and $c,d\geq 1$.
If $a=2$ the result follows from Theorem~\ref{th:known}.4 and Lemma \ref{1^234}.
Suppose now $a=1$. If $b$ is even we apply Lemma \ref{134}, otherwise we apply Lemma \ref{1234}.
Finally, assume $a=0$. By \cite{HR09} we may assume $b\geq 1$. If $b$ is odd we apply Lemma \ref{234odd} and so,
we may also assume $b\geq 2$ is even.

We start with the $\{2,3,4\}$-growable cyclic realizations of $L=\{2^2, 3^c, 4^d\}$ for each of the $12$ possibilities of congruence class combinations of $(c,d)\pmod {(3,4)}$ described in the first part of Table \ref{Tab0234}.
Note that $c+d\geq 5$.
Using Theorem~\ref{th:multigrow},  this proves $\BHR(L)$  except for the following cases:
$d=1,2$; $d=3$ and $c\not \equiv 0\pmod{3}$; $d=4$ and $c\equiv 1 \pmod{3}$.
For these exceptions, Table \ref{Tab0234} also gives $\{2,3\}$-growable cyclic realizations
when $(c,d)\in \{(4,1), (5,1),(6,1),(3,2),  (4,2),(5,2), (2,3),  (4,3),(4,4) \}$ and a
$2$-growable  cyclic realization for  $\{ 2^2, 3, 4^4\}$.

We are left to the cases $c+d\leq 4$.
This table also provides a $2$-growable cyclic realization for the multiset
$\{2^4, 3^c,4^d\}$ when $(c,d)\in \{(1,2),(1,3),(2,1),(2,2),(3,1)  \}$,
and for the multiset $\{2^6, 3,4\}$.
\end{proof}

\begin{table}[ht]
\caption{$\{2,3,4\}$-growable cyclic realizations for $\{2^b, 3^c,4^d\}$, with $b\geq 2$ even:
they are $x$-growable at $m_x$. The congruence classes of $(c,d)$ are taken  modulo $(3,4)$.}\label{Tab0234}
\begin{center}
\begin{footnotesize}
\begin{tabular}{lllll}\hline
Classes & Realizations & $(b,c,d)$ & $(m_2,m_3,m_4)$ & Missing cases\\  \hline
$(0,0)$ & $[ 2, 5, 9, 6, 8, 0, 4, 1, 7, 3 ]$ & $( 2, 3, 4)$  & $(1,2,3)$ \\
$(0,1)$ & $[ 8, 0, 7, 3, 10, 1, 5, 2, 9, 6, 4 ]$ & $(2, 3, 5)$ & $(7,3,4)$ & $d=1$ \\
$(0,2)$ & $[ 8, 0, 4, 2, 10, 1, 5, 7, 11, 3, 6, 9 ]$ & $(2, 3, 6)$ & $(7,8,3)$ & $d=2$  \\
$(0,3)$ & $[ 3, 6, 1, 5, 8, 2, 4, 0, 7 ]$ & $(2, 3, 3)$ & $(6,2,3)$  \\
$(1,0)$ & $[ 4, 8, 0, 2, 6, 10, 1, 9, 5, 3, 11, 7 ]$ & $( 2, 1, 8 )$ & $(3,6,7)$ & $ d=4$ \\
$(1,1)$ & $[ 4, 8, 5, 1, 6, 2, 0, 7, 3 ]$ & $(2, 1, 5)$ & $(2,3,4)$  & $d=1$ \\
$(1,2)$ & $[ 4, 8, 0, 2, 6, 9, 5, 1, 7, 3 ]$ & $( 2, 1, 6)$ & $(2,3,5)$ & $d=2$  \\
$(1,3)$ & $[ 4, 8, 1, 5, 7, 0, 9, 2, 6, 3, 10 ]$ & $( 2, 1, 7)$ & $(8,3,5)$  & $d=3$ \\
$(2,0)$ & $[ 2, 6, 0, 3, 7, 5, 1, 8, 4 ]$ & $( 2, 2, 4)$ &  $(1,3,4)$ \\
$(2,1)$ & $[ 4, 7, 1, 5, 3, 9, 6, 2, 8, 0 ]$ & $(2, 2, 5)$ & $(7,3,4)$  & $d=1$\\
$(2,2)$ & $[ 5, 8, 10, 3, 7, 0, 4, 6, 2, 9, 1 ]$ & $(2, 2, 6)$ & $(8,4,5)$  & $d=2$ \\
$(2,3)$ & $[ 10, 2, 6, 8, 0, 9, 5, 1, 11, 3, 7, 4 ]$ & $( 2, 2, 7 )$ & $(9,3,6)$ & $d=3$  \\\hline
$(0,1)$ & $[ 4, 7, 9, 6, 3, 0, 8, 1, 5, 2 ]$ & $( 2, 6, 1)$ & $(7,3,4)$ \\
$(0,2)$ & $[ 2, 5, 7, 3, 6, 0, 4, 1 ]$ & $( 2, 3, 2 )$ & $(2,3,-)$ \\
$(1,0)$ & $[ 0, 3, 7, 10, 8, 1, 5, 2, 9, 6, 4 ]$ & $(2, 4, 4)$ & $(7,3,4)$ & $(c,d)=(1,4)$ \\
$(1,1)$ & $[ 0, 3, 1, 6, 2, 5, 7, 4 ]$ & $( 2, 4, 1 )$  & $(5,2,-)$ \\
$(1,2)$ & $[ 3, 6, 1, 4, 0, 7, 5, 2, 8 ]$ & $( 2, 4, 2 )$ & $(6,2,3)$ \\
$(1,3)$ & $[ 8, 1, 5, 2, 0, 7, 3, 9, 6, 4 ]$ & $( 2, 4, 3 )$ & $(7,3,4)$ \\
$(2,1)$ & $[ 7, 5, 2, 8, 1, 4, 0, 6, 3 ]$ & $(2, 5, 1 )$ & $(6,2,3)$ \\
$(2,2)$ & $[ 9, 1, 5, 2, 8, 6, 3, 0, 7, 4 ]$ & $( 2, 5, 2)$ & $(7,3,4)$ \\
$(2,3)$ & $[ 7, 3, 6, 4, 0, 2, 5, 1 ]$ & $(2, 2, 3)$ & $(3,4,-)$ \\ \hline
$(1,0)$ & $[ 0, 4, 2, 6, 1, 5, 3, 7 ]$ & $(2,1,4)$ & $(3,-,-)$\\\hline
$(0,1)$ & $[ 2, 5, 7, 0, 4, 1, 8, 6, 3 ]$ & $( 4, 3, 1)$ & $(6,2,3)$ \\
$(1,1)$ & $[7, 5, 2, 0, 4, 6, 8, 1, 3 ]$ & $( 6, 1, 1 )$ & $(6,2,-)$ \\
$(1,2)$ & $[ 7, 1, 3, 5, 2, 6, 4, 0 ]$ & $(4, 1, 2)$ & $(3,4,-)$ \\
$(1,3)$ & $[ 3, 7, 0, 4, 6, 8, 1, 5, 2 ]$ & $(4, 1, 3)$ & $(1,2,4)$ \\
$(2,1)$ & $[ 1, 3, 0, 6, 2, 4, 7, 5 ]$ & $( 4, 2, 1)$ & $(5,2,-)$ \\
$(2,2)$ & $[ 8, 6, 2, 0, 4, 1, 7, 5, 3 ]$ & $(4, 2, 2)$ & $(6,2,3)$ \\
\hline
\end{tabular}
\end{footnotesize}
\end{center}
\end{table}

The success in proving  these small cases leads us to make the following conjecture, which says that the method of the previous section and this one is always successful.

\begin{conj}\label{conj:grow}
For any fixed set~$U$, there is a finite set of growable realizations with underlying set~$U$ that implies the existence of realizations for all but finitely many admissible multisets~$L$ that have underlying set~$U$.
\end{conj}

If Conjecture~\ref{conj:grow} is true, then the BHR Conjecture for any given underlying set can be proved with a finite set of realizations.

\section{A Partial Solution for $U = \{ 1, x, 2x \}$}\label{sec:1_x_2x}

In previous sections we have seen how it is possible to completely prove the BHR Conjecture for a fixed~$U$ by the construction of one or more base case realizations for each of $\prod_{x \in U} x$ cases.   In this section we develop ways to produce  general results with fewer base cases. 

Our main goal is Theorem~\ref{th:1_x_2x}, which says that $\BHR(L)$ holds for $L = \{ 1^a, x^b, (2x)^c \}$ when $a \geq x-2$, $c$ is even, and $b \geq 5x-2+c/2$.  When~$x$ is even, this covers many instances not covered by Theorem~\ref{th:known}.8.  When~$x$ is odd, the instances covered are all new.

\begin{lem}\label{lem:x_2x}
Suppose $L$ has an $X$-growable realization and take $x \in X$.  Take~$i$ with $1 \leq i \leq x$. Then $L \cup \{ x^{3x-2i}, (2x)^{2i}  \}$ has an  $X$-growable realization.
\end{lem}

\begin{proof}
Apply Theorem~\ref{th:grow} three times to the $x$-growable realization of~$L$ to obtain an $X$-growable realization of $L \cup \{x^{3x} \}$ with each of the subsequences
\begin{flushleft}
$  [m, m+x, m+2x, m+3x], 
  [m-1, m-1+x, m-1+2x, m-1+3x], 
 \ldots, $
\end{flushleft} 
\begin{flushright}
$ [m-x+1, m+1, m+1+x, m+1+2x] $
\end{flushright}
appearing, possibly reversed.
Each subsequence has differences $\{ x^3 \}$.  Take~$i$ of the subsequences and in each switch the middle two elements (so, for example, the first would become $[m, m+2x, m+x, m+3x]$).  Each time we perform this operation we obtain a subsequence with differences $\{ x, (2x)^2 \}$ instead of  $\{ x^3 \}$.  After performing it $i$ times the new differences are $\{ x^{3x-2i}, (2x)^{2i}  \}$.

These operations do not interfere with growability: if the original realization is $y$-growable at~$m'$, then the new realization is $y$-growable at $m'$ if $m' \leq m$ and at $m' + 3x$ otherwise.
\end{proof}

Let  $L = \{ 1^a, x^b, (2x)^c \}$.
When $x=1$ or 2 $\BHR(L)$ follows from Theorem~\ref{th:known}.1 or~\ref{th:known}.2 respectively, so $x=3$  is the first open case.  We treat the $x=3$ case first both as an illustration of the general method and because some of the later constructions require $x > 3$.

\begin{lem}\label{lem:136}
Let $L = \{ 1^a, 3^b, 6^c \}$. If $c$ is even and $b \geq 13 + c/2$, then BHR($L$) holds; if $c$ is odd and  $b \geq 18 + (c-1)/2$, then BHR($L$) holds.
\end{lem}

\begin{proof}
By Theorem~\ref{th:known}.1 we may assume that $a,b,c \geq 1$.  The multiset~$L$ is not admissible when $a=1$ and $b+c \equiv 1 \pmod{3}$.

Table~\ref{tab:136} gives $\{1,3\}$-growable realizations for 
$$(a,b,c) \in \{ (2,4,0), (1,5,0), (1,6,0), (1,10,1), (1,11,1), (2,9,1) \}. $$  
First we use Lemma~\ref{lem:x_2x} along with the realizations of Table~\ref{tab:136} to obtain a $\{1,3\}$-growable realization of 
$L'=\{1^{a'}, 3^{b'}, 6^{c'} \}$ with $a' \in \{1,2\}$, $b' \equiv b \pmod{3}$ and $c' \leq 5$ such that $c' \equiv c \pmod{6}$. 
We have to distinguish cases according
to the congruence class of $c$ modulo $6$.

\begin{table}[tp]
\caption{$\{1,3\}$-growable realizations of $\{ 1^a , 3^b, 6^c \}$.
Where they are $1$- or $3$-growable is indicated by $(m_1, m_3 )$.  }\label{tab:136}
$$
\begin{array}{llll}
\hline
\text{name} & \text{realization} & (a,b,c) & (m_1, m_3) \\
\hline
{\bm g_1} & [ 6, 5, 1, 4, 0, 3, 2 ] & (2,4,0) & (4,2) \\
{\bm g_2} & [ 3, 0, 6, 2, 5, 1, 4 ] & (1,5,0) & (5,2) \\
{\bm g_3} & [ 5, 2, 7, 0, 3, 6, 1, 4 ] & (1,6,0) & (6,2) \\

{\bm g_4} & [ 2, 12, 9, 6, 3, 0, 10, 7, 4, 5, 8, 1, 11 ] & (1,10,1) & (3,5) \\
{\bm g_5} & [ 5, 2, 13, 10, 7, 8, 11, 0, 3, 6, 9, 12, 4, 1 ] & (1,11,1) & (6,8) \\
{\bm g_6} & [ 9, 6, 3, 0, 10, 7, 8, 11, 1, 4, 5, 12, 2 ] & (2, 9,1) & (6,8) \\

\hline
\end{array}
$$

\end{table}

When~$c$ is even, we start with  ${\bm g_1}$, ${\bm g_2}$ or  ${\bm g_3}$.
If $c \equiv 0 \pmod{6}$, then start  by taking ${\bm g_1}$, ${\bm g_2}$ or ${\bm g_3}$ according to whether $b$ is congruent to 1, 2 or $0 \pmod{3}$ respectively.
If $c \equiv 2 \pmod{6}$, then start by taking ${\bm g_1}$, ${\bm g_2}$ or ${\bm g_3}$ according to whether $b$ is congruent to 2, 0 or $1 \pmod{3}$ respectively and apply Lemma~\ref{lem:x_2x} with $i=1$. 
If $c \equiv 4 \pmod{6}$, then start  by taking ${\bm g_1}$, ${\bm g_2}$ or ${\bm g_3}$ according to whether $b$ is congruent to 0, 1 or $2 \pmod{3}$ respectively and apply Lemma~\ref{lem:x_2x} with $i=2$.  

When~$c$ is odd, we start with  ${\bm g_4}$, ${\bm g_5}$ or  ${\bm g_6}$.
If $c \equiv 1 \pmod{6}$, then start  by taking ${\bm g_4}$, ${\bm g_5}$ or ${\bm g_6}$ according to whether $b$ is congruent to 1, 2 or $0 \pmod{3}$ respectively.
If $c \equiv 3 \pmod{6}$, then start  by taking ${\bm g_4}$, ${\bm g_5}$ or ${\bm g_6}$ according to whether $b$ is congruent to 2, 0 or $1 \pmod{3}$ respectively and apply Lemma~\ref{lem:x_2x} with $i=1$.
If $c \equiv 5 \pmod{6}$, then start  by taking ${\bm g_4}$, ${\bm g_5}$ or ${\bm g_6}$ according to whether $b$ is congruent to 0, 1 or $2 \pmod{3}$ respectively and apply Lemma~\ref{lem:x_2x} with $i=2$.  

In each case we obtain the required realization of $L'$.
Next, apply Lemma~\ref{lem:x_2x} $(c-c')/6$ times with $x=i=3$ to obtain a $\{1,3\}$-growable realization of $\{1^{a'}, 3^{b' + (c-c')/2}, 6^c \}$.
Finally, complete to the required realization using $a-a'$ applications of Theorem~\ref{th:grow} with $x=1$ and $\frac{b - b'}{3} -\frac{ c-c'}{6}$ applications with $x=3$.

When~$c$ is even, the method requires up to six 3's in the ${\bm g_i}$, up to seven 3's to adjust the congruency class of the number of 6's, and then $c/2$ 3's to obtain the correct number of 6's.  Hence it always works for $b \geq 6+7+c/2 = 13 + c/2$.  When $c$ is odd, the method requires up to eleven 3's in the ${\bm g_i}$, up to seven 3's to adjust the congruency class of the number of 6's, and then up to $(c-1)/2$ 3's to obtain the correct number of 6's.  Hence it always works for $b \geq 11+7+(c-1)/2 = 18+ (c-1)/2$.
\end{proof}

\begin{exa}
  Let $L=\{1^3,3^{18},6^{10}\}$. Since $b\equiv 0 \pmod 3$ and  $c\equiv 4 \pmod {6}$, we start applying Lemma~\ref{lem:x_2x} with $i=2$ to ${\bm g_1}$.
  In this way we obtain the realization  
  $$[15,14,1,\textrm{7},\textrm{4},10,13,0,\textrm{6},\textrm{3},9,12,11,8,5,2]$$
 of the multiset $\{1^2,3^9,6^4\}$.
  Now we apply Lemma~\ref{lem:x_2x} once with $i=x=3$ to this new multiset and we get a realization of $\{1^2,3^{12},6^{10}\}$:
  $$[24,23,1,7,4,10,\textrm{16},\textrm{13},19,22,0,6,3,9,\textrm{15},\textrm{12},18,21,20,17,14,11,\textrm{5},\textrm{8},2].$$
 We now apply Theorem~\ref{th:grow} twice with $x=3$ to get a realization of $\{1^2,3^{18},6^{10}\}$ and then once with $x=1$ to get a realization of $L$:
\begin{flushleft}
$[31,30,1,4,7,13,10,16,22,19,25,28,29,0,3,6,12,$
\end{flushleft}
\begin{flushright}
$9,15,21,18,24,27,26,23,20,17,11,14,8,5,2].$
\end{flushright}

\end{exa}

To use the method of proof of Lemma~\ref{lem:136} for $x>3$ we require $\{1,x\}$-growable realizations for $\{1^a, x^b \}$ for $a$ as small as possible and for a $b$ in each congruence class modulo~$x$.  Lemmas~\ref{lem:1x_misc},~\ref{lem:1x_xeven} and~\ref{lem:1x_xodd} provide these.  Each of the constructions has at least one subsequence consisting of multiple instances of pairs $[t, t+x]$, triples $[t, t+x, t+2x]$ or their reverses; we indicate these pairs and triples with underbraces to help illuminate the overall structure.

\begin{lem}\label{lem:1x_misc}
Let $x \geq 4$.
The multisets $\{1^{x-1}, x^{x+1} \}$, $\{1^{x-2}, x^{x+2} \}$ and $\{1^{x-2}, x^{2x} \}$ have $\{1,x\}$-growable realizations.
\end{lem}

\begin{proof}
First, we cover $\{1^{x-1}, x^{x+1} \}$, in which case $v = 2x+1$.  When~$x$ is even, the sequence
$$[1,x+1, 0, 2x, x, \ub{x-1, 2x-1}, \ub{2x-2,x-2}, \ub{x-3,2x-3}, \ldots, \ub{x+2, 2}]$$
has edge-lengths
$[ x , x, 1, x, 1, x, \ldots, 1, x ]$
and so realizes $\{1^{x-1}, x^{x+1} \}$.  It is 1-growable at~1 and $x$-growable at~$x$.
When $x$ is odd, the sequence
$$ [ x, x+1, 1, 0, 2x, x-1, 2x-1, \ub{x-2, 2x-2}, \ub{2x-3, x-3}, \ub{x-4, 2x-4}, \ldots, \ub{x+2, 2} ] $$
has edge-lengths
$ [ 1, x, 1, 1, x, x, x, x,1, x, \ldots, 1,x  ] $
and so realizes $\{1^{x-1}, x^{x+1} \}$.  It is 1-growable at~$2x-1$ and $x$-growable at~$x-1$.

Next, consider $\{ 1^{x-2}, x^{x+2} \}$ and so $v = 2x+1$.  When~$x$ is even, the sequence
$$[ x, 2x, 0, x+1, 1, x+2, 2,  \ub{3,x+3}, \ub{x+4,4}, \ub{5,x+5}, \ldots, \ub{x-1, 2x-1} ]$$
has edge-lengths
$[ x, 1, x, x, x, x, 1,x,1,x,\ldots, 1,x]$
and so realizes $\{ 1^{x-2}, x^{x+2} \}$.  It is 1-growable at~1 and $x$-growable at~$x$.
When~$x$ is odd, the sequence
$$[0, x, x-1, 2x, 2x-1, x-2, 2x-2, \ub{x-3, 2x-3}, \ub{2x-4, x-4}, \ub{x-5, 2x-5}, \ldots, \ub{x+1, 1}]$$
has edge-lengths
$[ x, 1, x, 1, x, x,x,x, 1, x, 1, x, \ldots, 1, x]$
so this realizes $\{ 1^{x-2}, x^{x+2} \}$.
It is 1-growable at $2x-2$ and $x$-growable at $x-1$. 

Finally, consider $ \{1^{x-2}, x^{2x} \}$ and so $v = 3x-1$.  When~$x$ is even, the sequence
\begin{flushleft}
$[ \ub{0,x,2x}, \ub{2x+1,x+1,1}, \ub{2,x+2,2x+2}, \ldots, \ub{x-4,2x-4,3x-4},$
\end{flushleft}
\begin{flushright}
$ x-3, 2x-3, 2x-2, x-2, 3x-3, 3x-2,x-1, 2x-1]$
\end{flushright}
has edge-lengths
$[ x,x,1,x,x,1, \ldots, 1, x,x,  x,x,1,  x,x,1,x,x]$
and so realizes $\{1^{x-2},  x^{2x} \}$.  It is 1-growable at~$3x-3$ and $x$-growable at~$x-1$.
When $x$ is odd the sequence
\begin{flushleft}
$ [3x-2, x-1, 2x-1, \underbrace{2x, x, 0}, \underbrace{1, x+1, 2x+1}, \underbrace{2x+2, x+2, 2}, \ldots, $
\end{flushleft}
\begin{flushright}
$ \underbrace{x-4, 2x-4, 3x-4}, x-3, 2x-3, 2x-2, x-2, 3x-3]$
\end{flushright}
has edge-lengths
$[ x, x, 1, x, x, 1, x, x, \ldots, 1, x, x , x, x, 1, x, x  ]$
and so realizes $\{ 1^{x-2}, x^{2x} \}$.  It is 1-growable at $3x-4$ and $x$-growable at $x$. 
\end{proof}

\begin{exa}\label{ex:1x_misc}
Let $x=8$.  Lemma~\ref{lem:1x_misc} gives the $\{1,8\}$-growable realizations
$$ [1,9,0,16,8,7,15,14,6,5,13,12,4,3,11,10,2 ] , $$
\vspace{-7mm}
$$ [ 8,16,0,9,1,10,2,3,11,12,4,5,13,14,6,7,15  ] , $$
$$ [ 0,8,16,17,9,1,2,10,18,19,11,3,4,12,20,5,13,14,6,21,22,7,15] $$
of $\{1^7,8^9\}$, $\{1^6,8^{10}\}$ and $\{1^6, 8^{16}\}$ respectively.

Let $x=9$.  Lemma~\ref{lem:1x_misc} gives the $\{1,9\}$-growable realizations
$$[  9, 10, 1, 0, 18, 8, 17, 7, 16, 15, 6, 5, 14, 13, 4, 3, 12, 11, 2   ], $$
\vspace{-7mm}
$$[ 0,9,8,18,17,7,16,6,15,14,5,4,13,12,3,2,11,10,1 ], $$
$$[ 25, 8, 17, 18, 9, 0, 1, 10, 19, 20, 11, 2, 3, 12, 21, 22, 13, 4, 5, 14, 23, 6, 15, 16, 7, 24 ]$$
of $\{1^8,9^{10}\}$, $\{1^7,9^{11}\}$ and $\{1^7, 9^{18}\}$ respectively.
\end{exa}

\begin{lem}\label{lem:1x_xeven}
Let $x \geq 4$ be even.  There is a $\{1,x\}$-growable realization for $\{1^{x-2}, x^b \}$ for~$b$ in range $x+3 \leq b \leq 2x-1$.
\end{lem}

\begin{proof}
We consider odd~$b$ and even~$b$ separately, starting with odd~$b$.

Take $r$ in the range $0 \leq r \leq (x-4)/2$.  Write $x = 2r+2s+4$ for some $s \geq 0$.  We construct a realization for 
$$L = \{ 1^{2r+2s+2}  , x^{4r+2s+7} \} = \{1^{x-2}, x^{ x+ 2r+3} \}.$$
We have $v = (2r+2s+2) + (4r + 2s + 7) + 1 = 6r+4s+10$.

We build the required realization by concatenating three sequences.  First:
$$[ \ub{2r+1, 4r+2s+5}, \ub{4r+2s+6,2r+2}, \ub{2r+3,4r+2s+7}, \ldots, \ub{4r+4s+6, 2r+2s+2} ],$$
which has $2s+2$ pairs and produces edge-lengths~$\{ 1^{2s+1}, x^{2s+2} \}$.
Second:
\begin{flushleft}
$[ \ub{6r+4s+8, 4r+2s+4, 2r}, \ub{2r-1, 4r+2s+3, 6r+4s+7}, $
\end{flushleft}
\begin{flushright}
$ \ub{6r+4s+6, 4r+2s+2, 2r-2}, \ldots, \ub{4r+4s+8, 2r+2s+4, 0} ],$
\end{flushright}
which has $2r+1$ triples and produces~$\{1^{2r}, x^{4r+2} \}$.
Third:
$$[ 6r+4s+9, 2r+2s+3, 4r+4s+7 ]$$
which produces $\{ x^2 \}$.

Upon concatenation we have a difference of $x$ where the first and second sequences join and a difference of~1 where the second and third join.  Hence we have a realization of
$$L =  \{ 1^{2s+1}, x^{2s+2} \} \cup \{1^{2r}, x^{4r+2} \} \cup \{ x^2 \} \cup \{1,x\}  = \{ 1^{x-2}, x^{ x + 2r + 3} \}. $$

It is  1-growable at $v-2 = 6r+4s+8$: when embedding with $m=6r+4s+8$, the only lengthened edge is $(2r+2s+2, 6r+4s+8)$.  It is $x$-growable at $x-1 = 2r+2s+3$: when embedding with $m = 2r+2s+3$ the lengthened edges are $(i, i+x)$ for $0 \leq i \leq x-1$.

For the case with~$b$ even, first note that when~$x=4$ there are no values of~$b$ to be considered.
Let $x \geq 6$ be even and take $r$ in the range $0 \leq r \leq (x-6)/2$.  Write $x = 2r+2s+6$ for some $s \geq 0$.  We construct a realization for 
$$L = \{ 1^{2r+2s+4}  , x^{4r+2s+10} \} = \{1^{x-2}, x^{ x+ 2r+4} \}.$$
We have $v = (2r+2s+4) + (4r + 2s + 10) + 1 = 6r+4s+15$.

We build the required realization by concatenating three sequences.  First:
\begin{flushleft}
$[ \ub{4r+2s+11, 2r+5}, \ub{2r+6,4r+2s+12},$
\end{flushleft} 
\begin{flushright}
$\ub{4r+2s+13,2r+7}, \ldots, \ub{4r+4s+11, 2r+2s+5} ]$,
\end{flushright}
which has $2s+1$ pairs and produces the edge-lengths~$\{ 1^{2s}, x^{2s+1} \}$.
Second:
$$[2r+2s+6, 4r+4s+12, 4r+4s+13, 2r+2s+7, 1, 4r+2s+10, 2r+4, 2r+3, 4r+2s+9, 0  ]$$
which produces $\{ 1^2 , x^7  \}$.
Third:
\begin{flushleft}
$[ \ub{6r+4s+14, 4r+2s+8, 2r+2}, \ub{2r+1, 4r+2s+7, 6r+4s+13}, $
\end{flushleft}
\begin{flushright}
$ \ub{6r+4s+12, 4r+2s+6, 2r }, \ldots, \ub{4r+4s+14, 2r+2s+8, 2} ],$
\end{flushright}
which has $2r+1$ triples and produces~$\{1^{2r}, x^{4r+2} \}$.

Upon concatenation we have a difference of $1$ at each of the joins. Hence we have a realization of
$$L =  \{ 1^{2s}, x^{2s+1} \} \cup \{1^{2}, x^{7} \} \cup \{ 1^{2r}, x^{4r+2} \} \cup \{1^2\}  = \{ 1^{x-2}, x^{ x + 2r + 4} \}. $$

It is  1-growable at $1$: when embedding with $m=1$, the only lengthened edge is $(1,2r+2s+7)$.  It is $x$-growable at $x = 2r+2s+6$: when embedding with $m = 2r+2s+6$ the lengthened edges are $(i, i+x)$ for $1 \leq i \leq x$.
\end{proof}

\begin{exa}\label{ex:1x_xeven}
To construct a $\{1,8\}$-growable realization of  $\{1^6, 8^{13} \}$ using the proof of Lemma~\ref{lem:1x_xeven} we take
$r=s=1$ to obtain
$$[3,11,12,4,5,13,14,6,18,10,2,1,9,17,16,8,0,19,7,15],$$
which is $1$-growable at~$18$ and $8$-growable at~$7$.

To construct a $\{1,10\}$-growable realization of  $\{1^8, 10^{16} \}$ we take
$r=s=1$  to obtain
$$  [17,7,8,18,19,9,10,20,21,11,1,16,6,5,15,0,24,14,4,3,13,23,22,12,2] $$
which is $1$-growable at~$1$ and $10$-growable at~$10$.
\end{exa}

\begin{lem}\label{lem:1x_xodd}
Let $x \geq 5$ be odd.  There is a $\{1,x\}$-growable realization for $\{1^{x-2}, x^b \}$ for~$b$ in range $x+3 \leq b \leq 2x-1$.
\end{lem}

\begin{proof}
The constructions are similar to those of Lemma~\ref{lem:1x_xeven} in that they each are built from the concatenation of three sequences and we need to consider odd and even~$b$ separately.  We start with odd~$b$.

Take~$r$ in the range~$0 \leq r \leq (x-5)/2$.  
Write $x = 2r + 2s +5$ for some $s \geq 0$.  We construct a realization for 
$$L = \{ 1^{2r+2s+3}, (2r+2s+5)^{4r+2s+9} \} = \{1^{x-2}, x^{x+2r+4} \}.$$ 
We have $v = (2r+2s+3) + (4r+2s+9) + 1 = 6r + 4s + 13$.

The first sequence is 
\begin{flushleft}
$[ \ub{6r+4s+10, 4r+2s+5, 2r}, \ub{2r-1, 4r+2s+4, 6r+4s+9}, $
\end{flushleft}
\begin{flushright}
$ \ub{6r+4s+8, 4r+2s+3,2r-2}, \ldots, \ub{4r+4s+10,2r+2r+5, 0} ]$. 
\end{flushright}
There are $2r+1$ triples, so this sequence produces edge-lengths $\{1^{2r}, x^{4r+2} \}$.
The second sequence is
$$[ 6r+4s+12, 2r+2s+4, 4r+4s+9, 4r+4s+8, 2r+2s+3, 6r+4s+11 ].$$
This has internal differences 
$[ 2r+2s+5 , 2r+2s+5, 1, 2r+2s+5, 2r+2s+5 ] $
and so produces $\{ 1, x^4 \}$.
The third sequence is
$$[ \ub{4r+2s+6, 2r+1}, \ub{2r+2, 4r+2s+7}, \ub{4r+2s+8,2r+3}, \ldots, \ub{ 2r+2s+2, 4r+4s+7 }]. $$
There are $2s+2$ pairs, so this sequence produces $\{ 1^{2s+1}, x^{2s+2} \}$. 

Upon concatenation, we have a difference of 1 generated where the first and second sequences join and a difference of $2r+2s+5 = x$ where the second and third join.  Hence we have a realization of
$$L =  \{1^{2r}, x^{4r+2} \} \cup \{ 1, x^4 \} \cup \{ 1^{2s+1}, x^{2s+2} \} \cup \{1,x\} = \{ 1^{x-2} , x^{x+2r+4}  \}.$$

The realization is 1-growable at $v-2 = 6r+4s+11$: when embedding with $m=6r+4s+11$, the only lengthened edge is $(2r+2s+3, 6r+4s+11)$.  It is $x$-growable at $x-1 = 2r+2s+4$: when embedding with $m = 2r+2s+4$ the lengthened edges are $(i, i+x)$ for $0 \leq i \leq x-1$.

Moving to even~$b$, take~$r$ in the range~$0 \leq r \leq (x-5)/2$ and  
write $x = 2r + 2s +5$ for some $s \geq 0$.  We construct a realization for 
$$L = \{ 1^{2r+2s+3}, (2r+2s+5)^{4r+2s+8} \} = \{1^{x-2}, x^{x+2r+3} \}.$$ 
We have $v = (2r+2s+3) + (4r+2s+8) + 1 = 6r + 4s + 12$.

The first sequence is
$$[ \ub{4r+2s+6, 2r+1}, \ub{2r+2,4r+2s+7}, \ub{4r+2s+8,2r+3}, \ldots, \ub{4r+4s+8,2r+2s+3}  ].$$
There are $2s+3$ pairs, so this sequence produces edge-lengths $\{ 1^{2s+2}, x^{2s+3} \}$. 
The second sequence is the same as the first sequence of the previous construction:
\begin{flushleft}
$[ \ub{6r+4s+10, 4r+2s+5, 2r}, \ub{2r-1, 4r+2s+4, 6r+4s+9}, $
\end{flushleft}
\begin{flushright}
$ \ub{6r+4s+8, 4r+2s+3,2r-2}, \ldots, \ub{4r+4s+10,2r+2r+5, 0} ]$. 
\end{flushright}
As before, there are $2r+1$ triples, so this sequence produces $\{1^{2r}, x^{4r+2} \}$.
The third sequence is
$$[ 6r+4s+11, 2r+2s+4, 4r+4s+9 ]$$
which has internal differences
$[ 2r+2s+5, 2r+2s+5 ]$
and so produces~$\{x^2 \}$.

Upon concatenation, we have a difference of $2r+2s+5 = x$ generated where the first and second sequences join and a difference of 1 where the second and third join.  Hence we have a realization of
$$L =  \{1^{2s+2}, x^{2s+3} \} \cup \{ 1^{2r}, x^{4r+2} \} \cup \{x^{2} \} \cup \{1,x\} = \{ 1^{x-2} , x^{x+2r+3}  \}.$$

The realization is 1-growable at $v-2 = 6r+4s+10$: when embedding with $m=6r+4s+10$, the only lengthened edge is $(2r+2s+3, 6r+4s+10)$.  It is $x$-growable at $x-1 = 2r+2s+4$: when embedding with $m = 2r+2s+4$ the lengthened edges are $(i, i+x)$ for $0 \leq i \leq x-1$.
\end{proof}

\begin{exa}\label{ex:1x_xodd}
To construct a $\{1,13\}$-growable realization of $\{1^{11},13^{21}\}$ using the proof of Lemma~\ref{lem:1x_xodd} we take
$r=s=2$ to obtain
\begin{flushleft}
$[ 30,17,4,3,16,29,28,15,2,1,14,27,26,13,0,  $
\end{flushleft}
\begin{flushright}
$ 32,12,25,24,11,31,18,5,6,19,20,7,8,21,22,9,10,23]$,
\end{flushright}  
which is $1$-growable at~$31$ and $13$-growable at~$12$.

To construct a $\{1,13\}$-growable realization of $\{1^{7},9^{14} \}$ we take
$r=s=1$ to obtain
$$[ 12, 3, 4, 13, 14, 5, 6, 15, 16, 7, 20, 11, 2, 1, 10, 19, 18, 9, 0, 21, 8, 17 ]$$
which is $1$-growable at~$20$ and $9$-growable at~$8$.
\end{exa}

We can now prove the main result of the section.

\begin{thm}\label{th:1_x_2x}
Let $L = \{1^a, x^b, (2x)^c \}$.  If $a \geq x-2$, $c$ is even and $b \geq 5x-2+c/2$, then $\BHR(L)$ holds.
\end{thm}

\begin{proof}
If $x \leq 2$, then $\BHR(L)$ holds without restriction and the case $x=3$ is covered in Lemma~\ref{lem:136}, so assume $x \geq 4$.  We follow the method of proof of Lemma~\ref{lem:136}, with Lemmas~\ref{lem:1x_misc},~\ref{lem:1x_xeven} and~\ref{lem:1x_xodd} providing the realizations to get started.

Take~$i$ in the range $0 \leq i < c/2$ such that $2i \equiv c \pmod{2x}$.

To construct the required realization for~$L$, start with the realization of $\{1^{a'}, x^{b'} \}$ that has $b' \equiv b +2i  \pmod{x}$ given by Lemma~\ref{lem:1x_misc},~\ref{lem:1x_xeven} or~\ref{lem:1x_xodd}.  So $a' = x-2$, except when $b+2i \equiv 1 \pmod{x}$ and admissibility forces us to use $a' = x-1$.

If $c \not\equiv 0 \pmod{2x}$, then apply Lemma~\ref{lem:x_2x} using~$i$ to give a realization whose number of occurrences~$c'$ of $2x$ differs from~$c$ by a multiple of~$2x$ and whose number of occurrences of~$x$ differs from~$b$ by a multiple of~$x$.  (If $c \equiv 0 \pmod{2x}$, then this is already the case.)  

Apply Lemma~\ref{lem:x_2x} a further $(c-c')/2x$ times with~$i=x$ to obtain a $\{1,x\}$-growable realization of $\{1^{a'}, x^{b''}, (2x)^c \}$ where $b'' \equiv b \pmod{x}$.  Complete to the required realization using the appropriate number of applications of Theorem~\ref{th:grow} with $1$ and~$x$.

The method requires up to $2x$ occurrences of $x$ in the initial realization, up to $3x-2$ occurrences of~$x$ to adjust the congruency class of the number of occurrences of~$2x$, and $c/2$ occurrences of~$x$ to obtain the correct number of occurrences of~$2x$.  Hence it always works for $b \geq 2x+3x-2+c/2 = 5x-2 + c/2$.
\end{proof}

When~$c$ is odd, we are not aware of any reason why the same approach will not work.  However, without new ideas, it will take more work to get weaker results than in the even case.  This is because the starter realizations now need to be for $\{1^{x-2}, x^b, 2x \}$, which means that we must have $v \geq 4x$, compared to the constructions here which all have $v < 3x$.  As well as being larger, using the same approach as Lemmas~\ref{lem:1x_xeven} and~\ref{lem:1x_xodd} would probably take more cases to cover all required values of~$b$.  Some of these issues are already apparent in Lemma~\ref{lem:136}.

Lemma~\ref{lem:x_2x} can be thought of as combining the notion of growability with that of a particular perfect realization.  This can be generalized to other perfect realizations, which we now do.

For a multiset $L$, define $sL = \{ sy : y \in L \}$. When we apply Theorem~\ref{th:grow} $k$ times to an $x$-growable realization we produce a realization with the $x$ subsequences 
$$[ m+1-x, m+1, m+1+x, \ldots, m+1+ (k-1)x] +t$$
for $0 \leq t \leq x-1$.
If we have a perfect linear realization of length~$k$ of a multiset~$L$, then we can multiply each element by $x$ to get a sequence that realizes $xL$ and then take a translate of it to replace a subsequence of the above form.  

Lemma~\ref{lem:x_2x} uses this process with the perfect linear realization $[0, 2,1, 3]$ of $\{1,2^2 \}$.  In general the approach gives the following lemma. 

\begin{lem}\label{lem:perf_grow}
Let $L$ have an $X$-growable realization with $x \in X$.  Let $L_1, \ldots, L_x$ be multisets of size $k-1$ that have perfect linear realizations.  Then 
$$L \cup xL_1 \cup \cdots \cup xL_x$$ 
has an $X$-growable realization.
\end{lem}

If we use the perfect linear realization $[0,1]$ of $\{1\}$ in Lemma~\ref{lem:perf_grow}, then we end up back at Theorem~\ref{th:grow}.

\begin{exa}\label{ex:perf_grow}
Let $x \geq 3$.  In this section we have constructed $\{1,x\}$-growable realizations of the multisets~$\{1^{x-1}, x^{x+1} \}$ and $\{ 1^{x-2}, x^b \}$ for $x+2 \leq b \leq 2x$.  Let $c,d,e,f,g \geq 0$ with $c+d+e+f+g \equiv 0 \pmod{x}$.  Take $c$ copies of the perfect linear realization $[0,1,2,3,4,5]$ of $\{1^5\}$, $d$ copies of the perfect realization $[0,2,1,3,4,5]$ of $\{1^3,2^2\}$, $e$ copies of the perfect linear realization $[0,3,1,2,4,5]$ of $\{1^2,2^2,3\}$, $f$ copies of the perfect linear realization $[0,3,1,4,2,5]$ of $\{2^2,3^3\}$, and $g$ copies of the perfect linear realization $[0,2,4,1,3,5]$ of $\{2^4,3\}$.  Lemma~\ref{lem:perf_grow} proves $\BHR(L)$ for 
$$L = \{ 1^a, x^{b + 5c + 3d + 2e}  , (2x)^{2d+2e+2f+4g}, (3x)^{e+3f+g} \}$$ 
for $a \geq x-2$ and $b \geq x+1$.
\end{exa}

\section*{Acknowledgements}
The second and the third author were partially supported by INdAM-GNSAGA.

\end{document}